\documentclass[12pt,final]{amsart}
\usepackage{geometry}
\usepackage{graphicx}
\usepackage[colorlinks=true, citecolor=blue, linkcolor=blue, urlcolor=blue]{hyperref}

\usepackage{mathrsfs}

\newgeometry{asymmetric, centering}

\usepackage{ifdraft}
\ifoptionfinal{
\usepackage[disable]{todonotes}
}{
\usepackage[norefs, nocites]{refcheck}

\usepackage[notref, notcite]{showkeys}
\usepackage[bordercolor=white, color=white]{todonotes}
}

\makeatletter\providecommand\@dotsep{5}\def\listtodoname{List of Todos}\def\listoftodos{\hypersetup{linkcolor=black}\@starttoc{tdo}\listtodoname\hypersetup{linkcolor=blue}}\makeatother

\newtheorem{lemma}{Lemma}
\newtheorem{proposition}{Proposition}
\newtheorem{theorem}{Theorem}

\newtheorem{definition}{Definition}
\theoremstyle{remark} 
\newtheorem{example}{Example}
\newtheorem{remark}{Remark}

\def\C{\mathbb C}
\def\R{\mathbb R}

\def\N{\mathbb N}

\def\p{\partial}
\DeclareMathOperator{\linspan}{span}
\DeclareMathOperator{\supp}{supp}
\DeclareMathOperator{\id}{id}
\DeclareMathOperator{\WF}{WF}
\DeclareMathOperator{\ccl}{ccl}
\DeclareMathOperator{\singsupp}{singsupp}
\newcommand{\pair}[1]{\left\langle #1 \right\rangle}

\let\Re\relax
\DeclareMathOperator{\Re}{Re}

\date{Compiled \today}
\title[Detection of connections in cubic wave equations]{Detection of Hermitian connections in wave equations with cubic non-linearity}

\author[X. Chen]{Xi Chen}
\address{ Department of Pure Mathematics and Mathematical Statistics, University of Cambridge, Cambridge CB3 0WB, UK. {\it E-mail address: \bf \tt xi.chen@dpmms.cam.ac.uk}}

\author[M. Lassas]{Matti Lassas}
\address{Department of Mathematics and Statistics, University of Helsinki. {\it E-mail address: \bf \tt Matti.Lassas@helsinki.fi}}

\author[L. Oksanen]{Lauri Oksanen}
\address{Department of Mathematics, University College London, Gower Street, London WC1E 6BT, UK. {\it E-mail address: \bf \tt l.oksanen@ucl.ac.uk}}

\author[G.P. Paternain]{Gabriel P. Paternain}
\address{ Department of Pure Mathematics and Mathematical Statistics, University of Cambridge, Cambridge CB3 0WB, UK. {\it E-mail address: \bf \tt g.p.paternain@dpmms.cam.ac.uk}}

\dedicatory{To the memory of our friend and colleague Slava Kurylev}

\begin{document}
\begin{abstract} We consider the geometric non-linear inverse problem of recovering a Hermitian connection $A$ from the source-to-solution map of the cubic wave equation $\Box_{A}\phi+\kappa |\phi|^{2}\phi=f$, where $\kappa\neq 0$ and $\Box_{A}$ is the connection wave operator in Minkowski space $\mathbb{R}^{1+3}$. The equation arises naturally when considering the Yang-Mills-Higgs equations with Mexican hat type potentials.
Our proof exploits the microlocal analysis of nonlinear wave interactions, but instead of employing information contained in the geometry of the wave front sets as in  previous literature, we study the principal symbols of waves generated by suitable interactions. Moreover, our approach relies on inversion of a novel non-abelian broken light ray transform. 
\end{abstract}
\maketitle

\tableofcontents


\section{Introduction}

This paper considers an inverse problem for a non-linear wave equation motivated by theoretical physics and differential geometry. The main problem we wish to address is the following: 
can the geometric structures governing the wave propagation 
be globally determined from local information,
or more physically, 
can an observer do local measurements to determine the geometric structures in the maximal region where the waves can propagate and return back? 
There has been recent progress on this question when the geometric structure is space-time itself and the 
relevant PDEs are the Einstein equations \cite{KLOU}. 

Here we propose the study of a natural non-linear wave equation when the Lorentzian background is fixed and the goal is the reconstruction of a Hermitian connection.
The main difference between the inverse problems for the 
Einstein equations and the equation considered here is that, in the former case, the geometric structure (the metric) to be reconstructed appears in the leading order terms, and in the latter case, it (the connection) appears in the lower order terms. 
This difference poses novel challenges, since a perturbation in the
leading order affects the wave front sets of solutions whereas lower order perturbations do not. 

The leading order terms can frequently be reconstructed via study of distances (or
time separations/earliest arrival times), whereas lower order terms often require reductions to light ray transform questions. 
Nevertheless, our approach exploits the recent philosophy that non-linear interaction of waves creates new singularities and enriches the dynamics \cite{KLOU, KLU, LUW}. As we shall see this interaction leads to a broken non-abelian light ray transform on lightlike geodesics that has not been previously studied.

Our main long term goal is the study of inverse problems for the Yang-Mills-Higgs equation. The present paper is the first stepping stone in this direction and our objective here is to start exposing the main features that this problem will have by considering a simplified, but non-trivial model case.

\subsection{The Yang-Mills-Higgs equations} 

Let $(M,g)$ be a Lorentzian manifold  of dimension $1+3$ and consider a compact Lie group $G$ with Lie algebra $\mathfrak{g}$. We choose a positive definite inner product on $\mathfrak{g}$ invariant under the adjoint action.
To simplify the exposition we discuss the case of 
the trivial bundle $M\times \mathfrak{g}$ over $M$.
The Yang-Mills-Higgs equations are PDEs on a pair $(A,\Phi)$, where 
$$
(A,\Phi)\in C^{\infty}(M;T^*M\otimes \mathfrak{g})\oplus C^{\infty}(M;\mathfrak{g}).
$$ 
Since the bundle $M\times \mathfrak{g}$ is trivial, $A$ is a connection, and it is called the Yang-Mills potential; $\Phi$ is the Higgs field.
The Yang-Mills-Higgs equations are
\begin{align}
&D_{A}^*F_{A}+[\Phi,D_{A}\Phi]=0;\label{eq:1}\\
&D_{A}^*D_{A}\Phi +V'(|\Phi|^{2})\Phi=0,\label{eq:2}
\end{align}
where $F_{A}:=dA+A\wedge A$ is the curvature of $A$,  $D_{A}\Phi:=d\Phi+[A,\Phi]$ is the associated covariant derivative, and $V'$ is the derivative of a smooth function $V: [0,\infty)\to\mathbb{R}$. The adoint $D_{A}^*$ is taken with respect to $g$ and hence $\Box_{A}:=D_{A}^*D_{A}$ is the wave operator associated with $g$ and $A$. 

An extensively studied  case is the  Yang-Mills-Higgs equations with the Mexican hat type potential,
    \begin{align}\label{eq: Mexican hat 1}
V(|\Phi|^{2})=\frac 12 \kappa(|\Phi|^{2}-
b)^2,
    \end{align}
where $\kappa,b\in \R$, see e.g.\ \cite[Eq. (10.5)]{Cottingham} where the Lagrangian formulation of the problem is used. We will consider the potential (\ref{eq: Mexican hat 1}) with $\kappa \ne 0$, and to simplify the notations, with $b = 0$. The case $b \ne 0$ is not substantially different. 
Our choice can be viewed as the simplest potential introducing a non-linearity. 
We refer also to 
\cite{Uhlenbeck}
where Yang-Mills-Higgs equations, with the  potential (\ref{eq: Mexican hat 1}), are discussed in a purely mathematical context, $(M,g)$ being a Riemannian manifold there.

As it is well known, equations (\ref{eq:1})--(\ref{eq:2}) are invariant under the group of gauge transformations which in this case coincides with the set of maps $\mathbf{u}\in C^{\infty}(M;G)$ and the action on pairs is 
$$(A,\Phi)\mapsto (\mathbf{u}^{-1}d\mathbf{u}+\mathbf{u}^{-1}A\mathbf{u}, \mathbf{u}^{-1}\Phi \mathbf{u}).$$ When $\Phi=0$ we obtain the pure Yang-Mills equation $D_{A}^*F_{A}=0$.

\subsection{Formulation of the inverse problem in the model case}

Dealing with the equations \eqref{eq:1}--\eqref{eq:2} from the outset might be too ambitious, so here we propose a simplified model. We shall suppose that we have a trivial bundle $E=M\times \mathbb{C}^{n}$
and a Hermitian connection $A$ on $E$ giving rise to a covariant derivative $d+A$. In this case, the gauge group is $U(n)$. 
We take $V$ to be the Mexican hat type potential (\ref{eq: Mexican hat 1}) with $b=0$, discard equation \eqref{eq:1} completely and focus on the analogue of equation \eqref{eq:2}, with 
$M\times \mathfrak{g}$ replaced by $E$. That is, we consider the equation
\begin{equation}\label{eq:model}
\Box_{A}\phi+\kappa|\phi|^{2}\phi=0,
\end{equation}
where $\phi$ is a section of $E$, $\Box_{A}=(d+A)^{*}(d+A)$ and $|\phi|$ is the norm with respect to the standard Hermitian inner product of $\mathbb{C}^n$. We shall further simplify matters by assuming that $M$ is $\mathbb{R}^{1+3}$ and that $g$ is the Minkowski metric.

Let us consider Cartesian coordinates $(t=x^{0},x^{1},x^{2},x^{3})$ on $\mathbb{R}^{1+3}$.
Let $\epsilon_0 > 0$ and define $B(\epsilon_0) = \{y \in \R^3 : |y| < \epsilon_0 \}$ and 
    \begin{align}\label{def_mho}
\mho = (0,1) \times B(\epsilon_0).
    \end{align}
We will formulate a source-to-solution map, associated to (\ref{eq:model}), that corresponds physically to measurements gathered in $\mho$. 
We can think that the measurements are performed by an observer travelling along the path 
\begin{align}\label{def_mu}
\mu : [0,1] \to \R^{1+3}, \quad \mu(t) = (t,0).
\end{align}
In what follows, we consider only finite time intervals, and write $\mathcal M = (-1, 2) \times \R^3$.
and let $\mathcal C$ be a small neighbourhood of the zero section in $C_0^4(\mho; E)$.
Then the source-to-solution map 
\begin{align}\label{def_L}
L_{A}f:= \phi|_\mho, \quad f \in \mathcal C,
\end{align}
is well-defined, where $\phi$ is the solution of 
\begin{align}\label{wave_nonlin}
&\Box_{A} \phi +\kappa |\phi|^2 \phi = f, \quad \text{in $\mathcal M$},
\\\notag
&\phi|_{t < 0} = 0.
\end{align}
We discuss the existence of $L_A$ in more detail in Section \ref{sec_prelim} below. 

The goal of the observer is to determine the Yang-Mills potential $A$ up to the natural obstructions, given the source-to-solution map $L_A$.
The causal structure of $(M,g)$ encodes the finite speed of propagation for the wave equation (\ref{eq:model}). 
Given $x,y\in M$ we say that $x \le y$ if $x=y$ or $x$ can be joined to $y$ by a future pointing causal curve,
and denote the causal future of $x\in M$ by $\mathscr J^{+}(x)=\{y\in M:\;x\leq y\}$. The causal future $\mathscr J^{+}(x)$ is the largest set that waves generated at $x$ can reach. The causal past of a point $z \in M$ is denoted by $\mathscr J^{-}(z)=\{y\in M:\;y\leq z\}$.
If waves generated at $x$ are recorded at $z$, the finite speed of propagation dictates that no information on the potential $A$ outside the causal diamond $\mathscr J^{+}(x) \cap \mathscr J^{-}(z)$ is obtained. 

{\em The model problem} is to determine $A$ given $L_{A}$
in the largest domain possible, that is, in 
    \begin{align}\label{def_diamond}
\mathbb{D} = \bigcup_{x,z \in \mho} \left(\mathscr J^+(x) \cap \mathscr J^-(z) \right),
    \end{align}
up to the natural gauge,
    \begin{align}\label{the_gauge}
A \mapsto \mathbf{u}^{-1}d\mathbf{u}+\mathbf{u}^{-1}A\mathbf{u},
    \end{align}
where $\mathbf{u} \in C^\infty(\mathbb D; U(n))$ and $\mathbf{u}|_\mho = \id$.
The sets $\mho$ and $\mathbb D$ are visualized in Figure \ref{fig_diamond}.

Observe that if we have two connections $A$ and $B$ on $M$ such that there exists a smooth map $\mathbf{u}:\mathcal M\to U(n)$ with the property that $B=\mathbf{u}^{-1}d\mathbf{u}+\mathbf{u}^{-1}A\mathbf{u}$ and $\mathbf{u}|_{\mho}=\text{id}$,
then $\Box_{B}=\mathbf{u}^{-1}\Box_{A}\mathbf{u}$ and $|\mathbf{u}\phi|=|\phi|$.
Moreover, as $f$ has compact support in $\mho$ it holds that $\mathbf{u}f=f$.
Therefore $\phi$ solves \eqref{wave_nonlin} for $B$ if and only if $\mathbf{u}\phi$ solves \eqref{wave_nonlin} for $A$, 
and it follows that $L_{A}=L_{B}$. This shows that the gauge (\ref{the_gauge}) is indeed natural.

Our main theorem asserts that the model problem has a unique solution, or in more physical terms, the measurements performed on $\mho$, as encoded by $L_A$, 
determine the gauge equivalence class of the Yang-Mills potential $A$, in the largest possible causal diamond $\mathbb D$. As $\mathbb D$ is strictly larger than $\mho$, we can view the determination of the equivalence class of $A$ as a form of remote sensing. We emphasize that the gauge equivalence classes of Yang-Mills potentials, not the potentials themselves, correspond to physically distinct configurations.

\begin{theorem}\label{thm:main}
Let $A$ and $B$ be two connections in $\R^{1+3}$ such that $L_{A}=L_{B}$ where the source-to-solution map $L_A$ is defined as above, and $L_B$ is defined analogously, with $A$ replaced by $B$ in \eqref{wave_nonlin}. Suppose that $\kappa\neq 0$ in \eqref{wave_nonlin}.
Then there exists a smooth $\mathbf{u}:\mathbb{D}\to U(n)$ such that $\mathbf{u}|_{\mho}=\id$ and $B=\mathbf{u}^{-1}d\mathbf{u}+\mathbf{u}^{-1}A\mathbf{u}$.
\end{theorem}

It is straightforward to see that $L_A = L_B$ implies $A = B$ on $\mho$. The non-trivial content of the theorem is the gauge equivalence away from $\mho$. To see that $A$ and $B$ coincide on $\mho$, we 
fix $y\in\mho$ and choose $\phi \in C_0^\infty(\mho; E)$ such that $\phi(y)=0$. Then for small $\epsilon > 0$
it holds that $f:= \epsilon(\Box_{A}\phi+\kappa |\phi|^{2}\phi) \in \mathcal C$. Since $L_{A}=L_{B}$ we see that at $y$:
\[-A_{0}\partial_t \phi +\sum_{j=1}^{3}A_{j}
\partial_{x^{j}} \phi =-B_{0} \partial_t \phi +\sum_{j=1}^{3}B_{j}\partial_{x^{j}} \phi,\]
cf. (\ref{boxA_mink}) below,
and since $d \phi$ at $y$ is arbitrary, there holds $A=B$ at $y$. 

\subsection{Comparison with previous literature}

The previous results on inverse problems for non-linear wave
equations, such as \cite{KLOU, KLU, LUW}, are based on analysis of four
singular, interacting waves. A new feature in the present paper is
that we consider interactions of three waves only. This leads to a
more economic proof, and is particularly well suited for the cubic
non-linearity in \eqref{wave_nonlin}.
A more detailed comparison of interaction three versus four waves is given in the beginning of Section \ref{sec_linalg_lemma}.

Let us briefly explain what we mean by the interactions of three waves.
The idea is to choose a source of the form $f = \epsilon_1 f_1 + \epsilon_2 f_2 + \epsilon_3 f_3$ 
where $\epsilon_j > 0$ are small and $f_j$ are conormal distributions. 
Then the cross-derivative $\partial_{\epsilon_1} \partial_{\epsilon_2} \partial_{\epsilon_3} \phi|_{\epsilon = 0}$ satisfies a linear wave equation with a right-hand side that corresponds to a certain product of $\partial_{\epsilon_j} \phi|_{\epsilon = 0}$, $j=1,2,3$. Here $\epsilon = (\epsilon_1, \epsilon_2, \epsilon_3)$. As also the functions $\partial_{\epsilon_j} \phi|_{\epsilon = 0}$ satisfy the linear wave equation, we can view the cross-derivative as a result of their interaction.

The wave front set of the above cross-derivative was studied in the case of the $1+2$-dimensional Minkowski space by Rauch and Reed \cite{Rauch1982}, see also \cite{Bony,Melrose1985} for later results of similar nature. 
What is new in the present paper, is that, contrary to \cite{Rauch1982} and the previous results on inverse problems for non-linear wave equations,
e.g. \cite{KLOU, KLU, LUW}, 
we employ more precise information on the singular structure of the cross derivative than just its wave front set.
Namely, in a suitable microlocal sense, the cross-derivative has a principal symbol, and the proof of Theorem \ref{thm:main} uses information contained in the principal symbol in order to recover a novel broken non-abelian light ray transform of the connection $A$ along lightlike geodesics.
The proof of Theorem \ref{thm:main} is completed by solving the subsidiary geometric inverse problem of inverting this transform in the Minkowski space, see Proposition \ref{prop_S} below, a result which has independent interest.

In more physical terms, 
we can say that the interaction of the three waves $\partial_{\epsilon_j} \phi|_{\epsilon = 0}$, $j=1,2,3$,
produce an artificial source, that can be viewed either as two moving point sources or as a filament in spacetime, and that emits a wave encoded by the cross-derivative $\partial_{\epsilon_1} \partial_{\epsilon_2} \partial_{\epsilon_3} \phi|_{\epsilon = 0}$. We show that, when the sources $f_j$, $j=1,2,3$, are chosen carefully, the singular wave front emitted by the artificial source returns to $\mho$. This wave front is visualized in Figure \ref{fig_3waves}. Stretching the physical analogy further, we can say the leading amplitude of this singular wave front is the information used in the proof.

\begin{figure}
\includegraphics[width=0.45\textwidth]{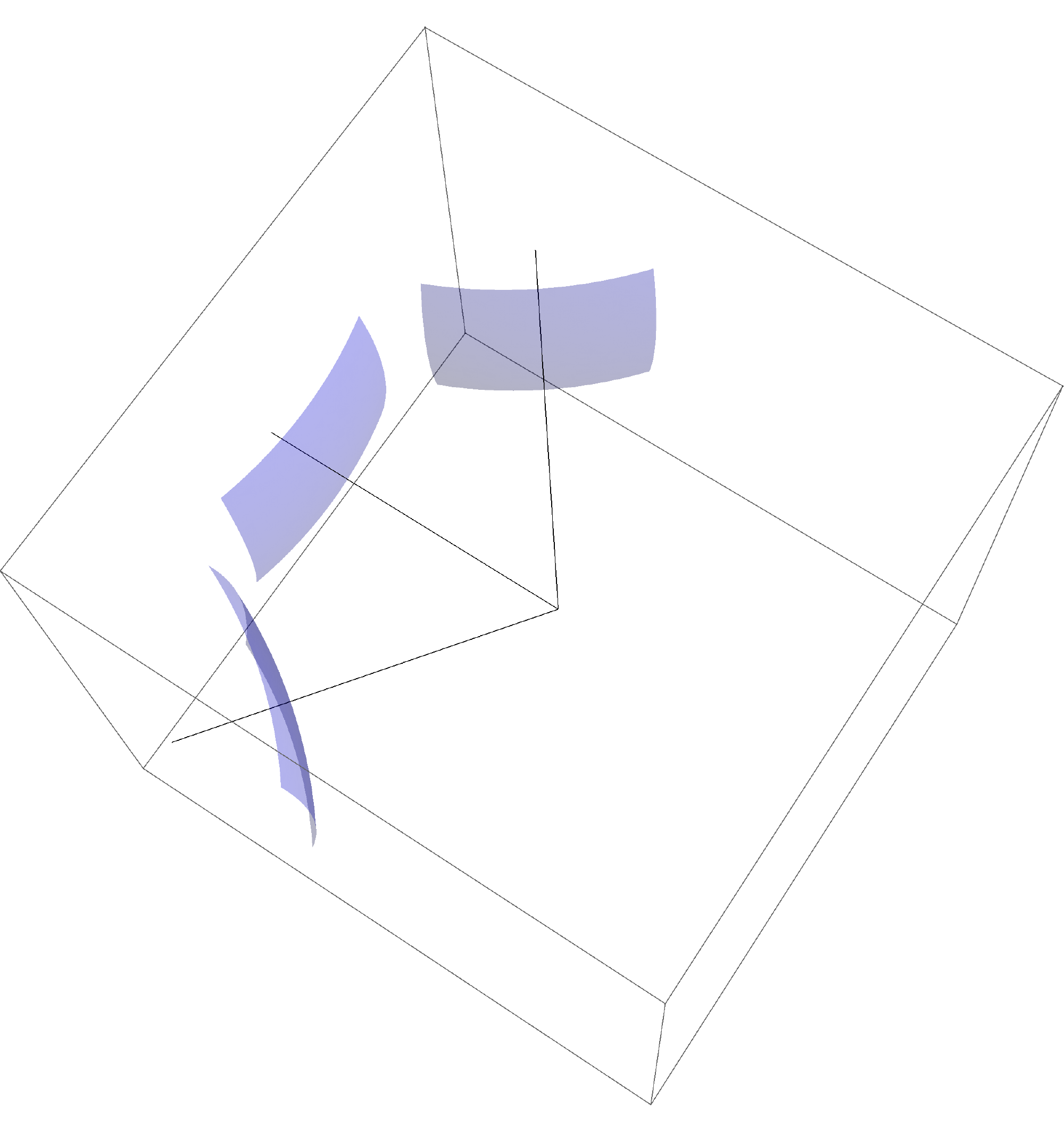}
\includegraphics[width=0.45\textwidth]{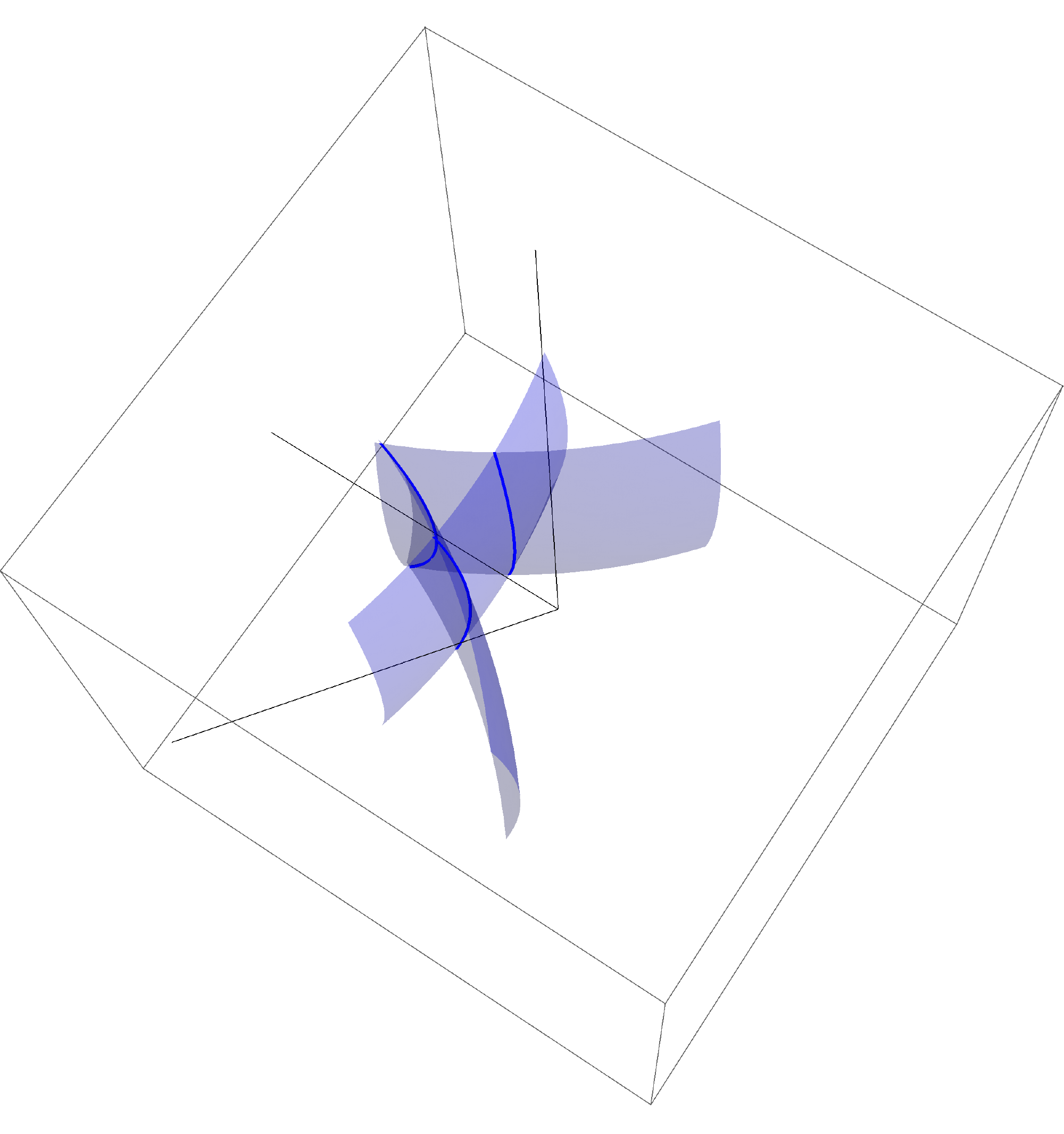}
\includegraphics[width=0.45\textwidth]{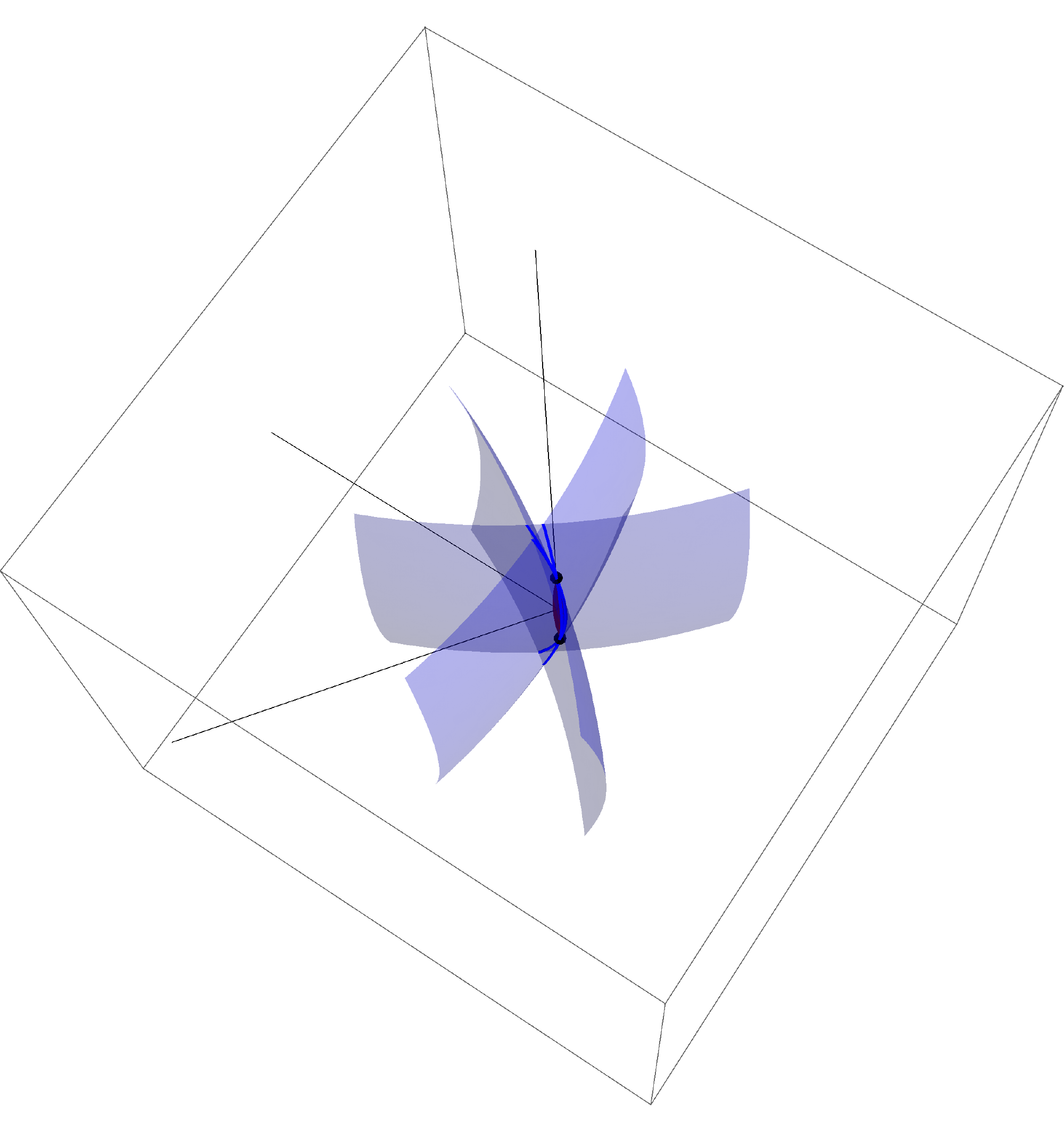}
\includegraphics[width=0.45\textwidth]{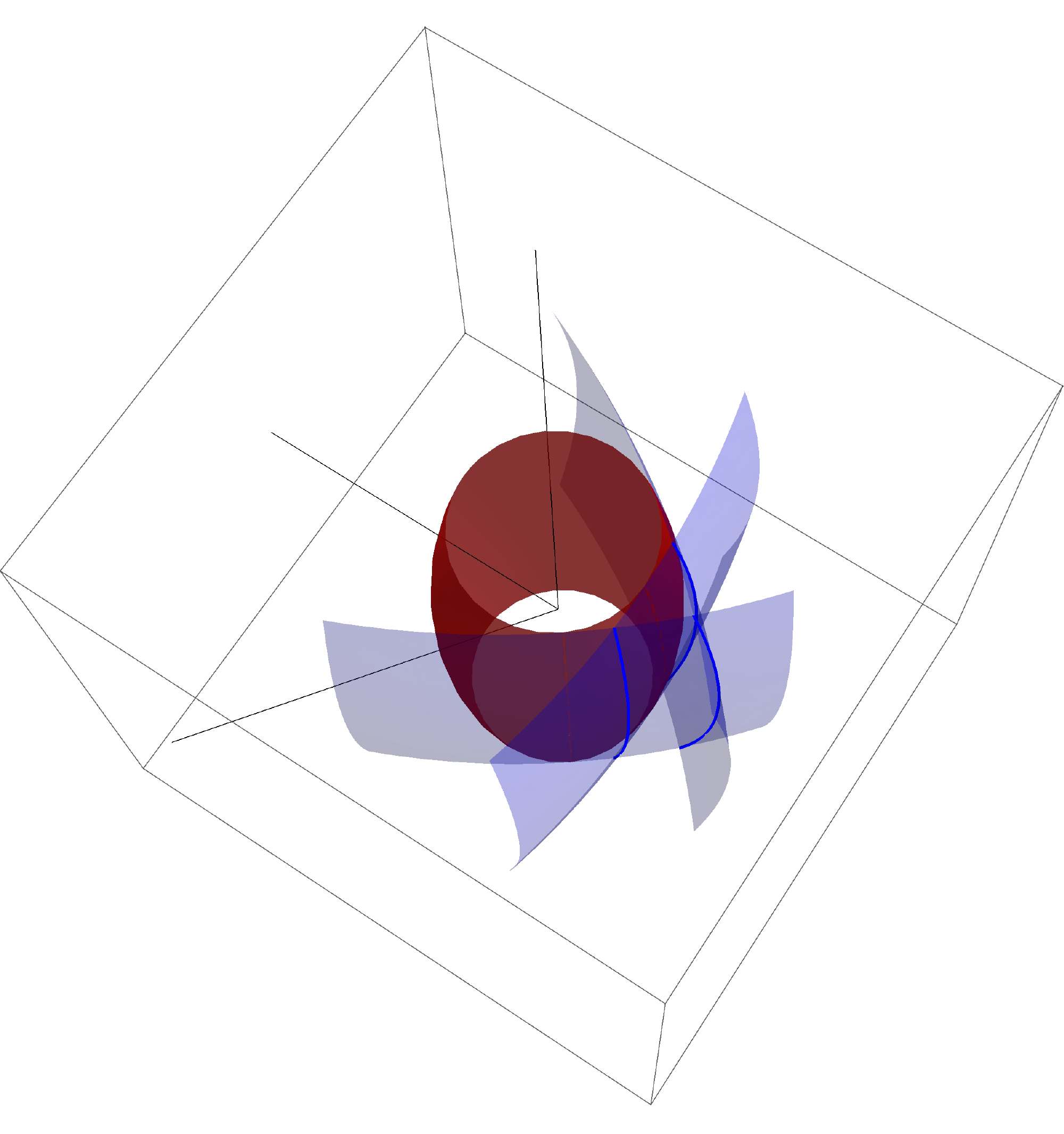}
\caption{
The interaction of three pieces of spherical waves, the 
blue surfaces. 
{\em Top~left.} The pieces propagate along the black lines, and do not intersect yet. {\em Top~right.} The pieces intersect along the blue curves. Pairwise intersections do not produce new propagating wave fronts. 
{\em Bottom left.} All the three pieces intersect in the two
black points, moving along the vertical axis over the
point where the black lines intersect. The points act as artificial sources
that produce a new propagating wave, the red surface. The line
segment traced by the two points can be viewed also as the projection
of a one dimensional filament acting as an artifical source. The
filament curves in spacetime since the two points move with a
non-constant speed.
{\em Bottom right.} As time progresses, the red propagating wave front grows. Eventually it will reach the points where the pieces of the spherical waves originate from. 
}
\label{fig_3waves}
\end{figure}

Paradoxically, Theorem \ref{thm:main} is open for the linear case, $\kappa=0$, but a positive solution is known if $A$ and $B$ are supposed to be time-independent \cite{KOP}. 
In the time-dependent case there are results \cite{Salazar} available only in the abelian case of a line bundle, $n=1$, and it is an open problem if recovery of $A$ in the optimal causal diamond $\mathbb D$ is possible in this case.
Let us also mention that the linear, abelian, time-independent case has been studied extensively, see e.g. \cite{Anderson2004,Belishev1992,Katchalov2001}, but these results do not carry over to the time-dependent case. The reason for this is that they (and also above mentioned \cite{KOP}) are based on Tataru's unique 
continuation principle \cite{Tataru}, that again is known to fail for equations with time-dependent coefficients \cite{Alinhac1983}.

We emphasize once more that 
the focus of the current 
paper is on the recovery of the lower order terms in a non-linear wave equation. This makes definite progress towards Open Problem 5 in \cite{Matti}, and is different from the previous results where only the determination of the leading order terms is considered.
See for instance \cite{KLU, LUW} where the  determination of the metric tensor (or its conformal class) is studied for
scalar-valued non-linear equations, or \cite{KLOU,Lassas2017} 
 where the  determination of the metric tensor is studied for
the Einstein equations coupled with different matter field equations.
The difference between recovery of leading and lower order terms is reflected in the key novelty of our approach, namely, in the study of principal symbols instead of wave front sets.

Such difference is apparent 
also in the existing theory of inverse problems for linear wave equations.
The case of a linear wave equation 
with time independent coefficients, and with sources and observations in disjoint sets, illustrates this. 
In this case the theory is still under active development, and the best results available are very different for leading and lower order terms:
the recovery of the metric \cite{LO_Duke} is based on distance functions, whereas the recovery of the lower-order terms \cite{KKLO} is based on focussing of waves. The latter also requires additional convexity assumptions, that are not present in the former case. 

\subsection{A conjecture on higher order non-linearities}

One outcome of the current paper is the following emergent principle for dealing with inverse problems for waves with polynomial non-linearities using an approach similar to ours. Assume for simplicity that we are in the line bundle case (i.e. $n=1$) and consider an equation of the form
    \begin{align}\label{eq_emergent_prin}
\Box_{A}\phi+\mathcal \kappa \phi^{N}=f,
    \end{align}
where $\kappa\neq 0$. Then to recover $A$ (up to gauge) from a source-to-solution map it is necessary to consider the $J$-fold linearization of (\ref{eq_emergent_prin}), where 
    \begin{align}\label{lin_lower_bound}
J\geq \max\{3,N\}.
    \end{align}
The necessity is discussed further in Remarks \ref{rem_two_not_enough} and \ref{rem_lin_N} below.
We conjecture that (\ref{lin_lower_bound}) is also a sufficient condition, but the present paper establishes this only in the case $N = 3$.

\subsection{Outline of the paper}

This paper is organized as follows. Section \ref{sec_prelim} contains preliminaries mostly
having to do with the direct problem  \eqref{wave_nonlin}.
Section \ref{micro_3} contains the microlocal analysis for the interaction of three waves and shows that we can recover the broken non-abelian light ray transform along  lightlike geodesics from the knowledge of $L_{A}$.
Section \ref{sec_X-ray} solves the geometric inverse problem of determining $A$ up to gauge from the broken non-abelian light ray transform and completes the proof of Theorem \ref{thm:main}. Appendix \ref{appendix} recalls the theory of conormal and Intersecting Pair of Lagrangian (IPL) distributions; Appendix \ref{appendix2} contains certain technical details concerning symplectic transformations to a model pair of intersecting Lagrangians, and Maslov bundles;
and Appendix \ref{appendix3} gives full description of the wave front set of the cross-derivative $\partial_{\epsilon_1} \partial_{\epsilon_2} \partial_{\epsilon_3} \phi|_{\epsilon = 0}$, that is, the red surface in Figure \ref{fig_3waves}.

\bigskip
\noindent
{\em We would like to dedicate this paper to the memory of our friend and
colleague Slava Kurylev
who was instrumental in initiating the present line of research on
inverse problems for the Yang-Mills-Higgs equations.
}

\bigskip

\noindent {\bf Acknowledgements.} 
LO thanks Allan Greenleaf, Alexander Strohmaier and Gunther Uhlmann for discussions on microlocal analysis. 

ML was supported by Academy of Finland grants
320113 and 312119.
LO was supported by EPSRC grants EP/P01593X/1 and EP/R002207/1 and	
XC and GPP were supported by EPSRC grant EP/R001898/1. GPP thanks the University of Washington for hospitality while this work was in progress and the Leverhulme trust for financial support.

\section{Preliminaries}\label{sec_prelim}

In this section, to accommodate further work, we let $(M,g)$ be an arbitrary, globally hyperbolic Lorentzian manifold of dimension $1+m$. 
Also $E$ can be taken as an arbitrary Hermitian vector bundle over $M$.
Recall that a Lorentzian manifold $(M, g)$ is globally hyperbolic if 
there are no closed causal paths in $M$, and 
the causal diamond $\mathscr J^{+}(x) \cap \mathscr J^{-}(z)$ is compact for any pair of points $x,z \in M$, see \cite{Bernal2007}.
A globally hyperbolic manifold $(M,g)$ is isometric to a product manifold $\mathbb{R} \times M_0$ with the Lorentzian metric given by
\begin{equation}
\label{gl_hyp_form}
g = -c(t,x') dt^2 + g_0(t,x'),
\quad (t,x') \in \mathbb{R} \times M_0,
\end{equation} 
where $c : \mathbb{R} \times M_0 \rightarrow \mathbb{R}^+$ is smooth and $g_0$ is a Riemannian metric on $M_0$ depending smoothly on $t$, see \cite{Bernal2005}.
Moreover, the vector field $\p_t$ gives time-orientation on $M$.
To simplify the discussion, we make the further assumption that all the geodesics of $(M,g)$ are defined on the whole $\R$.

\subsection{Direct problem}

We write occasionally $\nabla = d + A$ for the covariant derivative associated to the connection $A$, and view it as a map 
$$\nabla : C^\infty(M; E) \longrightarrow C^\infty(M; T^\ast M \otimes E).
$$
Writing $g = g_{ij}dx^idx^j$ in coordinates, we denote by $|g|$ and $g^{ij}$ the determinant and inverse of $g_{ij}$, respectively.
Moreover, $A = A_j dx^j$ is a $1$-form, and each $A_j$ is a skew-Hermitian matrix.
Let us now write the wave operator $\Box_A = \nabla^* \nabla$ in coordinates. 
Consider compactly supported sections $\phi$ and $\psi = \psi_j dx^j$ of $E$ and $T^\ast M \otimes E$, respectively. 
Then
$$
\langle \nabla \phi, \psi \rangle_{L^2(M; T^\ast M \otimes E)} = \int_{M} g^{ij}\langle (\p_{x^i} \phi + A_i \phi), \psi_j\rangle_E\,dV_g,$$
where $\langle \cdot, \cdot \rangle_E$ is the inner product on $E$, and $dV_g$ denotes the volume form on $(M,g)$.
Integrating by parts and using the fact that $A$ is skew-Hermitian, we have
    \begin{align*}
\langle d\phi, \psi \rangle_{L^2(M; T^* M \otimes E)} 
&=-\int_{M} \langle  \phi, |g|^{-1/2} \partial_{x^i}\big(|g|^{1/2} g^{ij} \psi_j\big) \rangle_E\,dV_g,
\\
\langle A\phi, \psi \rangle_{L^2(M; T^* M \otimes E)} 
&= - \int_{M} \langle \phi, g^{ij} A_i  \psi_j\rangle_E\,dV_g.
    \end{align*}
Consequently, $\Box_A \phi$ takes the form $$ - |g|^{-1/2} \partial_{x^i}\big(|g|^{1/2} g^{ij} \partial_{x^j}\phi\big) - 2 g^{ij}A_i \partial_{x^j}\phi  - |g|^{-1/2} \partial_{x^i}\big(|g|^{1/2} g^{ij}A_j\big) \phi - g^{ij}A_i A_j\phi.$$

\begin{remark}
To prove Theorem \ref{thm:main}, we will need the operator $\Box_{A}$ exclusively in Minkowski space where we can explicitly write
    \begin{align}\label{boxA_mink}
\Box_{A}\phi=\Box \phi-2\left(-A_{0} \partial_t \phi +\sum_{j=1}^{3}A_{j} \partial_{x^{j}} \phi \right)+\left(\text{div}A+A_{0}^{2}-\sum_{j=1}^{3}A_{j}^{2}\right)\phi,
    \end{align}
where $\Box$ and $\text{div}$ are the usual wave operator and divergence, 
\[\Box \phi = \partial_{t}^2 \phi -\sum_{j=1}^{3}\partial_{x^j}^2 \phi,
\quad \text{div} A= \partial_t A_{0} -\sum_{j=1}^{3} \partial_{x^{j}} A_{j}.
\]
\end{remark}

Let $T>0$ and let us consider the following nonlinear Cauchy problem 
\begin{equation}\label{eq_wave_lin T}
\begin{cases}
\Box_A \phi(t,x) = H(t,x, \phi(t,x))+f(t,x), & \hbox{on}\,(-\infty,T)\times M_0,\\ \phi = 0, & \hbox{for}\,\, t < 0,
\end{cases}
\end{equation} 
where $H :(\R\times M_0) \times E \to E$ is a smooth map operating section-wise such that $H(t,x,0)=0$, and $f$ is a section of $E$. We will now give sufficient conditions on $f$ in order for (\ref{eq_wave_lin T}) to have a unique solution.

As the leading term of $\Box_A$ is simply the canonical wave operator on $(M,g)$, acting on each component of $\phi$, we can use the standard results for quasilinear hyperbolic equations to show existence and uniqueness of solutions to this Cauchy problem. See, for example, Theorem 6 of \cite{Kato1975} and its proof (with the notations explained in detail in Appendix C in \cite{KLU2}) or Theorems I-III and the proof of Lemma 2.7 in \cite{HKM}. By these results, we have that for an integer $r>m/2+2$ and any compact set $\mathcal K\subset (0,T) \times M_0$, there is $\epsilon_0>0$ such that 
for any 
$f\in C^r_0(\mathcal K)$ 
satisfying $\|f\|_{C^r(\mathcal K)}<\epsilon_0$, the initial value problem (\ref{eq_wave_lin T}) has a unique solution. 
Recall that  $m$ is the dimension of the underlying space $M_0$.
In particular, in the case of the Minkowski space $\R^{1+3}$, we may take $r = 4$, and see that source-to-solution map $L_A$ is well-defined by (\ref{def_L}).

\subsection{Notations for microlocal analysis}
\label{sec_muloc_notations}

For a conic Lagrangian submanifold 
$\Lambda_0 \subset T^*M \setminus 0$ and a vector bundle $\mathbb E$ over $M$, 
we denote by $I^{p}(M;\Lambda_0;\mathbb E)$ the space of Lagrangian distributions of order $p \in \R$ associated to $\Lambda_0$,
and taking values in $\mathbb E$. 
If $\Lambda_1 \subset T^*M \setminus 0$
is another conic Lagrangian submanifold intersecting $\Lambda_0$ cleanly, we denote by $I^{p}(M;\Lambda_0,\Lambda_1;\mathbb E)$ the space of Intersecting Pair of Lagrangian (IPL) distributions of order $p \in \R$ associated to $(\Lambda_0, \Lambda_1)$,
and taking values in $\mathbb E$. 

We use occasionally the notation $\pair{\cdot, \cdot}$ for the duality pairing between covectors and vectors, and 
    \begin{align*}
I(M;\Lambda_0;\mathbb E) = \bigcup_{p \in \R} I^p(M;\Lambda_0;\mathbb E), \quad
I(M;\Lambda_0,\Lambda_1;\mathbb E) = \bigcup_{p \in \R} I^{p}(M;\Lambda_0,\Lambda_1;\mathbb E).
    \end{align*}
If $\Lambda_0$ coincides with the conormal bundle 
    \begin{align*}
N^* K = \{ (x,\xi) \in T^* M : x \in K,\ \pair{\xi, v} = 0 \text{ for all $v \in T_x K$}\}
    \end{align*}
of a submanifold $K \subset M$ in the sense that $\Lambda_0 = N^* K \setminus 0$, then the distributions in $I(M;\Lambda_0;\mathbb E)$ are called conormal distributions. Although removing the zero section from $N^* K$, when considering it as a conic Lagrangian manifold, is somewhat awkward notationally, it is natural to consider $N^* K$ as a submanifold of $T^* M$, since then the fibres $N_x^* K \subset T_x^* M$, $x \in K$, are linear subspaces.
We recall the basic properties of conormal and IPL distributions in Appendix \ref{appendix} below.

The wave front set of a distribution $u \in \mathcal D'(M)$ is denoted by $\WF(u)$, see \cite[Def. 2.5.2]{Hormander-FIO1}. It is a subset of $T^* M \setminus 0$, and its projection on $M$ is called the singular support     $\singsupp(u)$ of $u$. The wave front set $\WF(u)$ is conical and closed in $T^* M \setminus 0$, and it is occasionally convenient to use the notation 
    \begin{align*}
\ccl \mathscr B = \{(x,\lambda \xi) \in T^* M \setminus 0 : (x, \xi) \in \overline{\mathscr B},\ \lambda > 0\}
    \end{align*}
for the conical closure of a set $\mathscr B \subset T^* M \setminus 0$. For $u \in I(M;\Lambda_0;\mathbb E)$
it holds that $\WF(u) \subset \Lambda_0$.

If $\mathscr K$ is the Schwartz kernel of a pseudodifferential operator $\chi$ on $M$,
then the projection of $\WF(\mathscr K) \subset (T^*M \setminus 0)^2$ on the first factor $T^*M\setminus 0$ is called the essential support of $\chi$. 
(As $\WF(\mathscr K)$ is contained in the conormal bundle of the diagonal $\{(x,y) \in M^2 : x = y\}$, the choice between the first and second factor makes no difference.)
Following \cite[p. 124]{Hormander-FIO1} we write $\WF(\chi)$ for this set. 

We denote by $\Omega^{1/2}$ the half-density bundle over $M$. When $\Lambda_0$ and $\Lambda_1 \setminus \Lambda_0$ coincide with conormal bundles, and $\mathbb E = E \otimes \Omega^{1/2}$, there is a coordinate invariant way to define the principal symbol $\sigma[u]$ of $u \in I(M;\Lambda_0;\mathbb E)$, respectively $u \in I(M;\Lambda_0,\Lambda_1;\mathbb E)$, as an equivalence class of sections of $E \otimes \Omega^{1/2}$ over $\Lambda_0$, respectively $\Lambda_1 \setminus \Lambda_0$.
We will not emphasize the difference between the equivalence class $\sigma[u]$ and a representative of it, and we will also use the same notation for the half-density bundles over $M$ and $\Lambda_j$, $j=0,1$. Let us remark that there is typically no natural way to relate these bundles. For example, while it is natural to use $|g|^{1/4}$ to trivialize $\Omega^{1/2}$ over $M$, the Lorentzian metric $g$ on $M$ typically does not induce a natural trivialization of $\Omega^{1/2}$ over $\Lambda_j$.

For IPL distributions in $I(M;\Lambda_0,\Lambda_1;\mathbb E)$, there is also a refined notion of principal symbol, with components on both $\Lambda_0$ and $\Lambda_1$. 
We will use the refined principal symbol only in Appendix \ref{appendix}. 
The notation $\sigma[\chi]$ is used also for the principal symbol of a pseudodifferential operator $\chi$ on $M$. In this case, $\sigma[\chi]$ is represented by a section of $T^* M \setminus 0$.

\subsection{Microlocal analysis of the wave operator}
\label{sec_muloc_linear}

It is convenient to rescale (\ref{wave_nonlin}), and consider the following non-linear wave operator
$$
Q_{0}(\phi) = \frac 1 2 (\Box_A \phi + \kappa | \phi |^2 \, \phi),
$$
where $|\cdot| = |\cdot|_E$ is the norm with respect to the inner product $\langle \cdot, \cdot \rangle = \langle \cdot, \cdot \rangle_E$.
In order to make use of the microlocal machinery developed in \cite{Duistermaat-Hormander-FIO2}, we conjugate the operator $Q_{0}$ with the half density $|g|^{1/4}$ and consider the operator $Q(u) := |g|^{1/4} Q_{0}(|g|^{-1/4} u)$ acting on the sections of $E \otimes \Omega^{1/2}$.
Writing $P u = |g|^{1/4} \Box_A (|g|^{-1/4} u) / 2$, the operator $Q$ reads  
    \begin{align}\label{def_Q}
Q(u) = Pu + \frac \kappa 2 \big| |g|^{-1/4} u \big|^2 u. 
    \end{align}
For the sake of convenience, we will slightly abuse the notation, and write 
$$
uv = |g|^{1/4} ((|g|^{-1/4} u)(|g|^{-1/4} v))
$$
for products of half-densities as functions. Then $Q(u) = Pu + \kappa | u |^2 u/2$.

Writing $\imath = \sqrt{-1}$, the full symbol of the operator $P$ reads $$P(x, \xi) = g^{ij}\xi_i\xi_j/2 +  \imath^{-1} (\partial_{x^i} g^{ij}/2 +  g^{ij}A_i )\xi_j  + \mbox{zeroth order terms},$$ 
and we see that $P$ has the following principal symbol $\sigma[P]$ and subprincipal symbol $\sigma_{sub}[P]$, in the sense of Duistermaat and H\"ormander \cite[p.189 (5.2.8)]{Duistermaat-Hormander-FIO2},
    \begin{align}\label{symbols_of_P}
\sigma[P](x,\xi) =  g^{ij}\xi_i\xi_j/2, \quad \sigma_{sub}[P](x,\xi) =  \imath^{-1} g^{ij}A_i \xi_j, \quad (x,\xi) \in T^* M.    
    \end{align}
We write also $\sigma[P] = \pair{\xi, \xi}_g / 2$
where $\pair{\xi, \xi}_g$ denotes the inner product with respect to $g$.
Let us remark that the subprincipal symbol transforming as a connection is discussed in \cite{Strohmaier} in the more general context of pseudodifferential operators on vector bundles. 

We denote by $H_P$ the Hamiltonian vector field  associated to $\sigma[P]$, and by $\Sigma(P)$ the characteristic set of $P$. That is,
    \begin{align}
    \label{def_H}
H_P &= g^{ij} \xi_j \p_{x^i} - \frac 1 2 (\p_{x^i} g^{jk}) \xi_j \xi_k \p_{\xi^i},
\\\notag
\Sigma(P) &= \{ (x,\xi) \in T^*M \setminus 0 : \pair{\xi, \xi}_g = 0\}.
    \end{align}
The covectors $\xi$ satisfying $\pair{\xi, \xi}_g = 0$ are called lightlike. 
We denote by $\Phi_s$, $s \in \R$, the flow of $H_P$, and define for a set $\mathscr B \subset \Sigma(P)$ the future flowout of $\mathscr B$ by
\begin{align}
\label{def_flowout}
\{ (y,\eta) \in \Sigma(P);\ 
(y,\eta) = \Phi_s(x,\xi),\ s \in \R,\ (x,\xi) \in \mathscr B,\ y \ge x \}.
\end{align}

Let us recall the parametrix construction for the linear wave equation
    \begin{align}\label{eq_wave_lin}
\begin{cases}
Pu = f, & \text{in $\R \times M_0$},
\\
u|_{t < 0} = 0,
\end{cases}
    \end{align}
that originates from \cite{Duistermaat-Hormander-FIO2}. 
We will follow the purely symbolic construction from \cite{Melrose-Uhlmann-CPAM1979}, the only difference being that $u$ is vector valued in our case. 
For the convenience of the reader, we give a proof of the below theorem in Appendix \ref{appendix}. 

\begin{theorem}\label{th_parametrix}
Let $f \in I^k(M; \Lambda_0; E \otimes \Omega^{1/2})$ for a conic Lagrangian $\Lambda_0 \subset T^* M \setminus 0$. 
Suppose that $H_P$ is nowhere tangent to $\Lambda_0$,
write $\mathscr B = \Lambda_0 \cap \Sigma(P)$, and denote by $\Lambda_1$ the future flowout of $\mathscr B$.
Then the solution $u$ of (\ref{eq_wave_lin}) is in $I^{k - 2 + 1/2}(M; \Lambda_0, \Lambda_1; E\otimes \Omega^{1/2} )$. Moreover, 
    \begin{align}\label{transport_Lie}
&(\mathscr{L}_{H_P} + \imath \sigma_{sub}[P]) \sigma[u] = 0, &\text{on $\Lambda_1 \setminus \Lambda_0$},
\\\label{transport_init}
&\sigma[u] = \mathscr R((\sigma[P])^{-1} \sigma[f]),
&\text{on $\mathscr B$},
    \end{align}
where $\mathscr{L}_{H_P}$ is the Lie derivative with respect to $H_P$, $\sigma[u]$ and $\sigma[f]$ are the principal symbols of $u$ and $f$ on $\Lambda_1$ and $\Lambda_0$, respectively, and $\mathscr R$ is a map, defined by (\ref{def_R}) in Appendix \ref{appendix} below, 
    \begin{align*}
\mathscr R: S^{k - 1/2 + n/4}(\Lambda_0 \setminus \partial \Lambda_1; \Omega^{1/2}) \longrightarrow S^{k + n/4}(\Lambda_1; \Omega^{1/2})|_{\partial\Lambda_1},
    \end{align*}  
that acts as a multiplication by a scalar on $E$.
Here it is assumed that $\Lambda_0$ and $\Lambda_1 \setminus \Lambda_0$ coincide with conormal bundles.
\end{theorem}

\subsection{Flowout from a point in the Minkowski space}

The following case will be of particular importance for us. We have also included a detailed discussion of Theorem \ref{th_parametrix} in the context of this example case in Appendix \ref{appendix2}.

\begin{example}\label{ex_f_delta}
Let $(M,g)$ be the $1+3$-dimensional Minkowski space,
and consider a lightlike vector $\xi^0 \in T_0 \R^{1+3} \setminus 0$ of the form $\xi^0 = (1,\theta_0)$ where $\theta_0$ is in the unit sphere
    \begin{align*}
S^2 = \{\theta \in \R^3 : |\theta| = 1\}.
    \end{align*}
Let $c \in E_0 \setminus 0$, that is, $c$ is a non-zero vector in the fibre of $E$ over the origin, and let $\chi$ be a pseudodifferential operator such that $\sigma[\chi] \ne 0$ near $\ccl \{(0,\xi^0)\}$. We define 
    \begin{align*}
f \in I(M; \Lambda_0; E \otimes \Omega^{1/2}), \quad f = c |g|^{1/4} \chi \delta,
    \end{align*}
where $\Lambda_0 = T_0 \R^{1+3} \setminus 0 = N^* \{0\} \setminus 0$ and $\delta$ is the Dirac delta distribution at the origin. 
The corresponding future flowout $\Lambda_1$
satisfies $\Lambda_1 \setminus \Lambda_0 \subset N^* K \setminus 0$ where 
    \begin{align}\label{K_delta}
K = \{(t, t \theta) \in \R^{1+3}: t > 0,\ \theta \in S^2\}
    \end{align}
is the future light-cone in the spacetime $\R^{1+3}$ emanating from the origin.
Letting $u$ be the solution of (\ref{eq_wave_lin}),
its restriction on $\R^{1+3} \setminus 0$ 
is a conormal distribution in $I(\R^{1+3} \setminus 0; N^* K \setminus 0; E \otimes \Omega^{1/2})$. 
As $\sigma[f](0,\xi^0) \ne 0$, Theorem \ref{th_parametrix}
implies that for all $s > 0$ it holds that $\sigma[u](\gamma(s), \xi^0) \ne 0$
where $\gamma(s) = (s,s\theta_0)$, and $\xi^0$ is viewed also as an element of $T_{\gamma(s)} \R^{1+3}$.
The smaller the essential support $\WF(\chi)$ is chosen around 
$$
\ccl(0,\pm\xi^0) := \ccl \{(0,\xi^0), (0,-\xi^0)\},
$$
the smaller is $\singsupp(u) \subset \overline K = K \cup \{0\}$ around 
$$
\gamma(\R_+) := \{\gamma(s) : s \ge 0\}.
$$

The pseudodifferential operator $\chi$ can be chosen for example as follows. Choose functions $\chi_1 \in C^\infty(S^2)$, $\chi_2 \in C^\infty(\R)$ and $\chi_3 \in C^\infty(\R^{1+3})$ such that $\chi_1(\theta_0) = 1$, $\chi_2(1) = 1$ and $\chi_3(0) = 1$.
Let also $q \in \R$.
Then, writing $\xi = (\xi_0, \xi')$
and $c(\xi) = (\sqrt 2 |\xi|)^{-1}$, with $|\xi|$ the Euclidean norm of $\xi$,
we define the function
    \begin{align}\label{ex_f_delta_chi}
\chi_0(x,\xi) = \chi_3(x) \chi_2(c(\xi) \xi_0 ) \chi_1( c(\xi) \xi') |\xi|^q.
    \end{align}
Now $\chi_0$ is positively homogeneous of degree $q$.
Choose, furthermore, $\chi_4 \in C_0^\infty(\R^{1+4})$ such that $\chi_4 = 1$ near the origin.
Then $(1-\chi_4(\xi)) \chi_0(x,\xi)$ is smooth also near $\xi=0$, and it is a symbol in the sense of \cite[Def. 1.1.1]{Hormander-FIO1}. Now we define a pseudodifferential operator by
$$
\chi u = \int_{\R^{1+4}} e^{\imath \xi x} (1-\chi_4(\xi)) \chi_0(x,\xi) \hat u(\xi) d\xi.
$$
Ignoring $2\pi$ factors, the full symbol of $\chi$ is simply $(1-\chi_4(\xi)) \chi_0(x,\xi)$, and the principal symbol $\sigma[\chi]$ is the corresponding equivalence class modulo symbols of one degree lower order. 

Choose now another $\tilde \chi_4 \in C_0^\infty(\R^{1+4})$ such that $\tilde \chi_4 = 1$ near the origin. Then 
$$
(1-\chi_4(\xi)) \chi_0(x,\xi)
- (1-\tilde \chi_4(\xi)) \chi_0(x,\xi)
= (\tilde \chi_4(\xi)-\chi_4(\xi)) \chi_0(x,\xi)
$$
is smooth near $\xi = 0$ since $\tilde \chi_4-\chi_4 = 0$ there. Moreover, it is a symbol of order $-\infty$ as both $\chi_4$ and $\tilde \chi_4$ are compactly supported. Therefore also $(1-\tilde \chi_4(\xi)) \chi_0(x,\xi)$ is a representative of $\sigma[\chi]$.
As the support of $\chi_4$ can be chosen arbitrarily small it makes sense to say that $\sigma[\chi] = \chi_0$
although this does not quite hold for any representative. Also for $q = 0$ it holds, in this sense, that 
$\sigma[\chi] = 1$ in $\ccl\{(0,\xi^0)\}$.

Observe that $\WF(\chi)$ becomes small around $\ccl\{(0,\xi^0)\}$ when the supports of $\chi_j$, $j=1,2,3$, are small around $\theta_0$, $1$, and $0$ respectively. Also $\chi u(x) = 0$ for any $x$ outside the support of $\chi_3$, and for any $k \in \N$ we can choose $q < 0$ negative enough, so that $\chi\delta \in H^k(\R^{1+3})$.

Let us check that indeed $\Lambda_1 \setminus \Lambda_0 = N^* K \setminus 0$.
The lightlike vectors in $T_0 \R^{1+3} \setminus 0$ are given by 
$(\lambda , \lambda \theta)$ with $\lambda \in \R \setminus 0$ and $\theta \in S^2$.
In the Minkowski space, the tangent-cotangent isomorphism corresponds to changing the sign of the first component. Therefore,
$$
\Lambda_1 \setminus \Lambda_0 
= \{(s\lambda, s \lambda\theta; -\lambda , \lambda \theta) \in \R^{1+3} \times \R^{1+3} : \lambda \in \R \setminus 0,\ \theta \in S^2,\ s\lambda > 0\}.
$$
We can also reparametrize 
$$
\Lambda_1 \setminus \Lambda_0 
= \{(t, t\theta; -\lambda , \lambda \theta) \in \R^{1+3} \times \R^{1+3} : \lambda \in \R \setminus 0,\ \theta \in S^2,\ t > 0\}.
$$
Clearly the projection of $\Lambda_1 \setminus \Lambda_0$ on the base space $\R^{1+3}$ is $K$.

For the convenience of the reader, we will still compute explicitly the conormal bundle of $K$.
Toward that end, we choose local coordinates $\R^2 \supset B \ni a \mapsto \Theta(a) \in S^2$ on $S^2$
and see that the tangent space of $K$ at $(t,\theta_0)$  is given by the range of
    \begin{align}\label{K_param}
\begin{pmatrix}
1 & 0
\\
\theta_0 & t\frac{d \Theta}{d a}
\end{pmatrix}.
    \end{align}
Writing $\theta_0^\perp = \{v \in \R^3 : v \cdot \theta_0 = 0\}$
where $\cdot$ denotes the Euclidean inner product, we have that $\theta_0^\perp$ is the range of $d\Theta/da$.
Writing $(\xi_0, \xi') \in N^*_{(t,\theta_0)} K$, this fibre is characterized by
$$
\xi_0 \delta t + \xi' \theta_0 \delta t + \xi' \delta \theta = 0, \quad \delta t \in \R,\ \delta \theta \in \theta_0^\perp.
$$
Taking first $\delta t = 1$ and $\delta \theta = 0$ 
we have $\xi_0 = - \xi' \theta_0$. Letting then $\delta \theta$ vary we see that $\xi' \perp \theta_0^\perp$, that is, $\xi' = \lambda \theta_0$ where $\lambda \in \R$. Here we can view $\theta_0 \in \R^3$ as a covector since the tangent-cotangent isomorphism in $\R^3$ is the identity. Hence
$$
N^*_{(t,\theta_0)} K \setminus 0
= \{ (-\lambda, \lambda \theta_0) : \lambda \in \R \setminus 0 \},
$$
and indeed $\Lambda_1 \setminus \Lambda_0 = N^*K \setminus 0$.

Observe that $\Lambda_1 \setminus \Lambda_0$
is embedded in the following smooth submanifold of $T^* M \setminus 0$, that is the flowout to both past and future
    \begin{align}\label{Lambda_1_pastfuture}
\hat\Lambda_1 = \{(t, t\theta; -\lambda , \lambda \theta) \in \R^{1+3} \times \R^{1+3} : \lambda \in \R \setminus 0,\ \theta \in S^2,\ t \in \R\}.
    \end{align}
Note that while $K$ is singular at $t=0$, we see that $\hat\Lambda_1$ is not by considering the derivative analogous to (\ref{K_param}),
$$
\begin{pmatrix}
1 & 0 & 0
\\
\theta_0 & t\frac{d \Theta}{d a} & 0
\\
0 & 0 & -1 
\\
0 & \lambda\frac{d \Theta}{d a} & \theta_0 
\end{pmatrix}.
$$
This matrix is injective since $\lambda \ne 0$.
\end{example}

We will next write the transport equation (\ref{transport_Lie}) for the principal symbol $\sigma[u]$ as a parallel transport equation with respect to the covariant derivative $\nabla$,
and we begin by discussing $\Omega^{1/2}$ over the flowout $\Lambda_1$.

\subsection{Trivialization of the half-density bundle over the flowout}
\label{sec_halfdens_flowout}

We want to trivialize $\Omega^{1/2}$ in a way that preserves homogeneity properties, as possessed for example by $\chi_0$ in (\ref{ex_f_delta_chi}). Let us point out that, even in the context of Example \ref{ex_f_delta}, there appears to be no canonical choice of a non-vanishing section of $\Omega^{1/2}$ over the conormal bundle $\Lambda_1 \setminus \Lambda_0 = N^* K \setminus 0$.
For example, the Sasaki metric on $T^* \R^{1+3}$, associated with the Minkowski metric, is degenerate when restricted on $N^* K \setminus 0$.

The submanifold $K$ in Example \ref{ex_f_delta} is of codimension one in $\R^{1+3}$. This holds in general in the sense that, if $\Lambda_0$ and $\Lambda_1$ in Theorem \ref{th_parametrix} satisfy 
    \begin{align}\label{def_K}
\Lambda_1 \setminus \Lambda_0 = N^* K \setminus 0
    \end{align}
for a submanifold $K \subset M$, then $K$ is of codimension one. This can be seen as follows. Observe first that $\Lambda_1 \subset \Sigma(P)$ simply because $\Lambda_1$ is the future flowout from  $\Lambda_0 \cap \Sigma(P)$. Therefore, for any $x \in K$, the fibre $N_x^* K$ can contain only lightlike vectors with respect to $g$. On the other hand, if $\xi^1, \xi^2 \in T_x^* M$ are lightlike and linearly independent, then their linear span satisfies, see e.g. \cite[Cor.\ 1.1.5]{SW},
    \begin{align}\label{span_of_two}
\{\xi \in \linspan(\xi^1, \xi^2) : \text{$\xi$ is lightlike} \} = \linspan(\xi^1) \cup \linspan(\xi^2).
    \end{align}
In particular $\linspan(\xi^1, \xi^2)$ contains vectors that are not lightlike, and therefore at most one of $\xi^j$, $j=1,2$, can belong to $N_x^* K$. This shows that $N_x^* K$ is of dimension one, or equivalently, $K$ is of codimension one in $M$.

We will trivialize $\Omega^{1/2}$ over $\Lambda_1$ by choosing a strictly positive half-density $\omega$ in $C^\infty(\Lambda_1; \Omega^{1/2})$ that is positively homogeneous of degree $1/2$. We begin by recalling the definition of positive homogeneity following \cite[p. 13]{Hormander-Vol4}. Let $\lambda \in \R \setminus 0$ and define 
$$
m_\lambda : T^* M \setminus 0 \to T^* M \setminus 0, \quad m_\lambda(x,\xi) = (x, \lambda \xi).
$$
Then $m_\lambda$ restricts as a map on $\Lambda_1$ in a natural way, and we denote the restriction still by $m_\lambda$.
The half-density $\omega$ is said to be positively homogeneous of degree $q \in \R$ if $m_\lambda^* \omega = \lambda^q \omega$ for all $\lambda > 0$.

If in local coordinates 
\begin{equation}\label{K_locally}
K = \{(x^0,x') \in \R^{1+m}: x^0 = 0\},
\end{equation} 
then $N^* K = \{(x^0,x'; \xi_0, \xi') \in \R^{1+m} \times \R^{1+m} : x^0 = 0,\ \xi' = 0\}$. We emphasize that, as the conormal bundle of $K$ coincides with the flowout $\Lambda_1$ in the sense of (\ref{def_K}), the local coordinate $x^0$ is not the time coordinate $t$ in (\ref{gl_hyp_form}). Here $(x,\xi) = (x^0,x'; \xi_0, \xi')$ are the induced coordinates on $T^* M$. Considering the restriction 
$$
m_\lambda :N^* K\setminus 0 \to N^* K\setminus 0, 
$$
the pullback $m_\lambda^* \omega$ is given by 
$$
m_\lambda^* \omega(x',\xi_0) 
= 
\left|\frac{d m_\lambda}{d (x',\xi_0)}\right|^{1/2}
\omega(x',\lambda \xi_0)
= \lambda^{1/2} \omega(x',\lambda \xi_0).
$$
Thus $\omega$ being positively homogeneous of degree $1/2$ means that 
$\omega(x',\lambda \xi_0) = \omega(x', \xi_0)$, moreover, $\omega$ being strictly positive means that $\omega(x',\xi_0) > 0$.
We will, in fact, choose a half-density $\omega$ that is also symmetric in the sense that 
    \begin{align}\label{omega_symmetry}
\omega(x',\lambda \xi_0) = \omega(x', \xi_0),
\quad \lambda \in \R \setminus 0,
    \end{align}
a coordinate invariant formulation of which reads $m_\lambda^* \omega = |\lambda|^{1/2} \omega$ for all $\lambda \in \R \setminus 0$.

In general, a strictly positive half-density $\omega \in C^\infty(\Lambda_1; \Omega^{1/2})$ satisfying (\ref{omega_symmetry}) can be constructed by choosing an auxiliary Riemannian metric on $M$, restricting the associated Sasaki metric $h$ on $T^* M \setminus 0$ to $\Lambda_1$, and taking $\omega = |h|^{1/4}$.
In particular, when $(M,g)$ is the Minkowski space, it feels natural to choose the Euclidean metric on $M$.

\subsection{Parallel transport equation for the principal symbol}

Let us fix a strictly positive half-density $\omega$ over $\overline{N^* K}$ satisfying (\ref{omega_symmetry}). 
Here the closure is taken in $T^* M$,
and we remark that, in view of (\ref{def_K}), $\Lambda_1 \subset \overline{N^* K}$.

We write $\sigma[u] = a \omega$ where 
$a$ is a section in $C^\infty(N^* K \setminus 0; E)$.
A computation in coordinates shows that $\mathscr{L}_{H_P}(a \omega) = (H_P a)\omega + a \mathscr{L}_{H_P} \omega$. Introducing the notation $\text{div}_\omega H_P = \omega^{-1} \mathscr{L}_{H_P} \omega,$ we write 
$$
\omega^{-1}\mathscr{L}_{H_P}(a \omega) = H_P a + a \text{div}_\omega H_P.
$$ 
We want to further rewrite this as a conjugated differentiation along bicharacteristics, that is, along the flow curves of $H_P$. Recall that 
$\Phi_s$, $s \in \R$, denotes the flow of $H_P$. Writing 
$\beta(s) = \Phi_s(x_0, \xi^0)$ for the bicharacteristic through $(x_0, \xi^0) \in N^*K \setminus 0$, we have 
$$
(H_P a)\circ \beta(s) = \partial_r a(\Phi_r(\beta(s)))|_{r=0} = \partial_s(a\circ\beta)(s).
$$ 
We further write $\alpha = a \circ \beta$, and define
    \begin{align}\label{def_rho}
\rho(s) = \int_0^s\rho'(s') ds', \quad \rho' = \text{div}_\omega H_P \circ \beta.
    \end{align}
Then
$$
\omega^{-1}\mathscr{L}_{H_P}(a \omega)\circ \beta = \partial_s\alpha + \rho'\alpha = e^{-\rho} \partial_s(e^\rho \alpha).
$$ 

We denote by $\gamma$ the projection of the bicharacteristic $\beta$ to the base manifold $M$, and by $\dot \gamma^*$ the tangent vector of $\gamma$ as a covector, that is, 
$\dot \gamma^*_i = g_{ij} \dot \gamma^j$.
It follows from (\ref{def_H}) that $\gamma$ is a geodesic of $(M,g)$, and that $\beta(s) = (\gamma(s), \dot \gamma^*(s))$. As $\beta(s) \in \Sigma(P)$, the geodesic $\gamma$ is lightlike. 
Moreover, using (\ref{symbols_of_P}),
$$
\imath (\sigma_{sub}[P] \circ \beta)(s) = g^{ij}A_i \dot \gamma_j^*(s) = A_i \dot{\gamma}^i(s) = \pair{A, \dot\gamma(s)},
$$ 
where $\pair{\cdot, \cdot}$ is again the duality pairing between covectors and vectors.

The covariant derivative on the bundle $E$ along the geodesic $\gamma(s)$ is given by 
$$
\nabla_{\dot \gamma} = 
\p_s + \pair{A, \dot\gamma(s)}.
$$
Therefore (\ref{transport_Lie}) along $\beta$ is equivalent with $e^{-\rho} \nabla_{\dot{\gamma}}(e^\rho\alpha) = 0$.
If we define the following  symbol along $\beta$ for conormal distributions in $I(M; N^\ast K \setminus 0; E\otimes \Omega^{1/2})$, 
$$
\sigma_{\omega, \beta}[\cdot] = e^\rho(\omega^{-1}\sigma[\cdot])\circ\beta,
$$ 
then we can rewrite (\ref{transport_Lie}) along $\beta$ as follows
\begin{equation}\label{eqn : parallel transport for symbols} \nabla_{\dot{\gamma}}\sigma_{\omega, \beta} [u] = 0.
\end{equation}
This is the parallel transport equation along $\gamma$ with respect to the connection $A$.

If $x = \gamma(s_0)$ and $y = \gamma(s_1)$ for some $s_0,s_1 \in \R$, then we write $\mathbf{P}^A_{y \gets x} : E_x\to E_y$
for the parallel transport map from $x$ to $y$ along $\gamma$.
That is, $\mathbf{P}^A_{y \gets x} v = w(s_1)$ where $w$ is the solution of 
\begin{equation}\label{parallel_transport}
\begin{cases}
\dot{w} + \pair{A, \dot{\gamma}} w=0, & \text{on $(s_0,s_1)$},
\\
w(s_0)=v.
\end{cases}
\end{equation}
In general, the map $\mathbf{P}^A_{y \gets x}$ depends on the geodesic $\gamma$ joining $x$ and $y$, but not on the parametrization of $\gamma$. We do not emphasize the dependency on $\gamma$ in our notation, since we are mainly interested in the Minkowski case, and in this case $\mathbf{P}^A_{y \gets x}$ depends only on the points $x$ and $y$. To summarize, writing $\xi = \dot\gamma^*(s_0)$ and $\eta = \dot\gamma^*(s_1)$, it follows from (\ref{eqn : parallel transport for symbols}) that 
    \begin{align}\label{P_sigma_u}
e^{\rho(s_1)}(\omega^{-1}\sigma[u])(y,\eta) 
= e^{\rho(s_0)}\mathbf{P}^A_{y \gets x} ( (\omega^{-1}\sigma[u])(x,\xi) ).
    \end{align}

\subsection{Positively homogeneous symbols}

We will now consider how (\ref{P_sigma_u}) changes under rescaling of $\xi \in N_x^* K \setminus 0$, assuming that $\sigma[u]$ is positively homogeneous of degree $q + 1/2 \in \R$ at $(x, \xi)$, that is, 
    \begin{align}\label{poshomog_at_x}
(\omega^{-1}\sigma[u])(x,\lambda\xi)
= \lambda^q (\omega^{-1}\sigma[u])(x,\xi),
\quad \lambda > 0. 
    \end{align}

\begin{proposition}\label{prop_P_rescaling}
Let $\Lambda_j$, $j=0,1$, and $u$ be as in Theorem \ref{th_parametrix}, and suppose that (\ref{def_K}) holds. Let $(x,\xi) \in N^* K \setminus 0$ and $(y,\eta) = \Phi_{s_1}(x,\xi)$ for some $s_1 \in \R$. Suppose that (\ref{poshomog_at_x}) holds at $(x,\pm \xi)$ and that $(y,\eta) \in N^* K \setminus 0$. Then
    \begin{align}\label{P_sigma_u_rescaled}
e^{\rho(s_1)}(\omega^{-1}\sigma[u])(y,\pm \lambda \eta) 
= \lambda^q \mathbf{P}^A_{y \gets x} ( (\omega^{-1}\sigma[u])(x, \pm \xi) ), \quad \quad \lambda > 0.
    \end{align}
\end{proposition}
Recall that $\mathscr B = \Lambda_0 \cap \Sigma(P) = \Lambda_0 \cap \Lambda_1$.
As the symbol $\sigma[u]$ on $\Lambda_1 \setminus \Lambda_0$ is smooth up to $\mathscr B$,
equation (\ref{P_sigma_u_rescaled}) holds also when $(x,\pm \xi) \in \mathscr B$. 
Then $(y,\eta) \in N^* K \setminus 0$ implies $x < y$, that is, the causal relation $x \le y$ holds and $x \ne y$.
Writing again $\gamma(s)$ for the projection of $\beta(s) = \Phi_s(x,\xi)$ to $M$, we have
$$
\gamma(0) = x < y = \gamma(s_1).
$$ 
Hence, if $\xi$ is future pointing then $s_1 > 0$, and if $\xi$ is past pointing then $s_1 < 0$.
We emphasize that also past pointing singularities are propagated forward in time by the wave equation (\ref{eq_wave_lin}).

\begin{proof}
For $\lambda \in \R \setminus 0$, the bicharacteristic through $(x,\lambda \xi)$ is given by 
$$
\beta_\lambda(s) = (\gamma(\lambda s), \lambda \dot \gamma^*(\lambda s)),
$$
where $\gamma$ denotes still the projection of $\beta = \beta_1$ to $M$.
Analogously to (\ref{def_rho}), we write  $\rho_\lambda' = \text{div}_\omega H_P \circ \beta_\lambda$ and $\rho_\lambda(s) = \int_0^s\rho_\lambda'(s') ds'$, with the usual convention that $\int_0^s\rho_\lambda'(s') ds' = - \int_s^0 \rho_\lambda'(s') ds'$ for $s < 0$.
By applying (\ref{P_sigma_u}) to $\beta_\lambda$,
we obtain 
    \begin{align*}
e^{\rho_\lambda(s_1/\lambda)}(\omega^{-1}\sigma[u])(y,\lambda \eta) 
&= \mathbf{P}^A_{y \gets x} ( (\omega^{-1}\sigma[u])(x,\lambda\xi) )
= |\lambda|^q \mathbf{P}^A_{y \gets x} ( (\omega^{-1}\sigma[u])(x,\pm\xi) ),
    \end{align*}
where the sign is the one in $\lambda = \pm |\lambda|$.

It remains to show that $e^{\rho_\lambda(s_1/\lambda)} = e^{\rho(s_1)}$. 
We denote the restriction of the flow $\Phi_s$ on $N^* K \setminus 0$ still by $\Phi_s$.
In coordinates satisfying (\ref{K_locally}), the Lie derivative $\mathscr{L}_{H_P} \omega$ is of the form
    \begin{align*}
\mathscr L_{H_P} \omega(x',\xi_0) 
&= 
\p_s \Phi_s^* \omega(x',\xi_0)|_{s = 0}
= 
\p_s \left(\left|\frac{d\Phi_s}{d(x',\xi_0)} \right|^{1/2} \omega(\Phi_s(x',\xi_0)) \right)\bigg|_{s = 0}
\\&=
\p_s \bigg(\left|\frac{d\Phi_s}{d(x',\xi_0)} \right|^{1/2}\bigg)\bigg|_{s = 0}\, \omega(x',\xi_0) 
+ \p_s\bigg( \omega(\Phi_s(x',\xi_0))\bigg)\bigg|_{s = 0},
    \end{align*}
where we used the fact that $\Phi_0 = \id$, and therefore $|d\Phi_0 / d(x',\xi_0) |^{1/2} = 1$.
We write 
    \begin{align*}
H_P(x',\xi_0) = H^0(x',\xi_0) \p_{\xi_0} + \sum_{i=1}^n H^i(x',\xi_0) \p_{x'^i}
    \end{align*}
for the Hamiltonian vector field $H_P$ on $N^* K \setminus 0$.
Then, see e.g. p. 418 of \cite{LS}, 
    \begin{equation*}
\p_s\bigg( \left|\frac{d\Phi_s}{d(x',\xi_0)} \right|^{1/2}\bigg) \bigg|_{s=0}
= \frac 1 2 \p_s\bigg(\left|\frac{d\Phi_s}{d(x',\xi_0)} \right|\bigg) \bigg|_{s=0}
= \frac 1 2 \left(\p_{\xi_0}H^0 + \sum_{i=1}^n \p_{x'^i} H^i \right).
    \end{equation*}
Equation (\ref{def_H}), together with the fact that $\xi' = 0$ on $N^* K$, implies that 
    \begin{align*}
\p_{\xi_0}H^0 + \sum_{j=1}^n \p_{x'^i} H^i
= - (\p_{x^0} g^{0 0}) \xi_0 + \sum_{i=1}^n (\p_{x^i} g^{0 i}) \xi_0.
    \end{align*}
In particular, 
    \begin{align*}
(\p_{\xi_0}H^0)(x', \lambda \xi_0) + \sum_{j=1}^n (\p_{x'^i} H^i)(x', \lambda \xi_0) = \lambda (\p_{\xi_0}H^0)(x', \xi_0) + \lambda \sum_{j=1}^n (\p_{x'^i} H^i)(x', \xi_0).
    \end{align*}
Moreover, as $\omega$ satisfies (\ref{omega_symmetry}),
    \begin{align*}
\p_s \omega(\Phi_s(x', \lambda \xi_0))|_{s = 0}
&= \p_s \omega(\gamma(\lambda s), \lambda \dot\gamma^*(\lambda s))|_{s = 0}
= \p_s \omega(\gamma(\lambda s), \dot\gamma^*(\lambda s))|_{s = 0}
\\&= \lambda \p_s \omega(\gamma(s), \dot\gamma^*(s))|_{s = 0}
= \lambda \p_s \omega(\Phi_s(x',\xi_0))|_{s = 0}.
    \end{align*}    
The above six equations imply that
$\mathscr L_{H_P} \omega(x',\lambda\xi_0) 
= \lambda\mathscr L_{H_P} \omega(x',\xi_0)$.
This again implies that
    \begin{align*}
\text{div}_\omega H_P(x',\lambda\xi_0) = \omega^{-1}(x',\lambda\xi_0) \mathscr{L}_{H_P} \omega(x',\lambda\xi_0)
= \lambda \text{div}_\omega H_P(x',\xi_0),
    \end{align*}    
and therefore the change of variables $s' = \lambda s$ gives
    \begin{align*}
\rho_\lambda(s_1/\lambda) =
\int_0^{s_1/\lambda} \text{div}_\omega H_P(\gamma(\lambda s), \lambda \dot\gamma^*(\lambda s)) ds
= \int_0^{s_1} \text{div}_\omega H_P(\gamma(s'),  \dot\gamma^*(s')) ds' = \rho(s_1).
    \end{align*}
\end{proof}

\section{Microlocal analysis of the interaction of three waves}\label{micro_3}

The core idea of the proof of Theorem \ref{thm:main} is to choose the source $f$ in (\ref{wave_nonlin}) as the weighted superposition of three singular sources,
    \begin{align}\label{def_f}
f(\epsilon) = \sum_{\mathbf{j} = 1}^3 \epsilon_\mathbf{j} f_\mathbf{j}, \quad \epsilon = (\epsilon_1, \epsilon_2, \epsilon_3),\ \epsilon_\mathbf{j} \in \R,
    \end{align}
where $f_\mathbf{j}$, $\mathbf{j} = 1, 2, 3$, are conormal distributions supported in $\mho$ satisfying
    \begin{align}\label{support condition}
    \supp(f_{\mathbf{j}})\cap \mathscr J^+(\supp(f_{\mathbf{k}}))=\emptyset,\quad\hbox{for }\mathbf j\not =\mathbf k.
        \end{align} 
Recall that $\mho$ is the set where the measurements are gathered, see the definition (\ref{def_L}) of the source-to-solution map $L_A$. 
Each $f_\mathbf{j}$ will be similar to the source $f$ in Example \ref{ex_f_delta}.
We denote by $\phi = \phi(\epsilon)$ the solution of (\ref{wave_nonlin}) with the source $f = f(\epsilon)$,
and write 
    \begin{align}\label{phi_in_out}
\phi_{in} = \p_{\epsilon_1} \phi |_{\epsilon = 0},
\quad \phi_{out} = \p_{\epsilon_1}\p_{\epsilon_2}\p_{\epsilon_3} \phi|_{\epsilon = 0}.
    \end{align}

We will choose $f_1$ so that $\phi_{in}$ is singular along a lightlike geodesic $\gamma_{in}$ intersecting $\mho$. Then we choose $f_2$ and $f_3$ to be such perturbations of $f_1$ that $\phi_{out}$
is singular along another lightlike geodesic $\gamma_{out}$ that intersects both $\gamma_{in}$ and $\mho$. Let us fix the parametrization of $\gamma_{in}$ and $\gamma_{out}$ so that 
    \begin{align}\label{gamma_in_out}
\gamma_{in}(0) = x \in \mho,\quad \gamma_{in}(s_{in}) = \gamma_{out}(0) = y,\quad \gamma_{out}(s_{out}) = z \in \mho
    \end{align}
for some $s_{in}, s_{out} \in \R$.
Taking into account the fact that the wave equation (\ref{wave_nonlin}) propagates singularities forward in time, we consider only the case $x < y < z$, see Figure \ref{fig_diamond}, left.  

\begin{figure}
\def\svgwidth{0.4\textwidth}
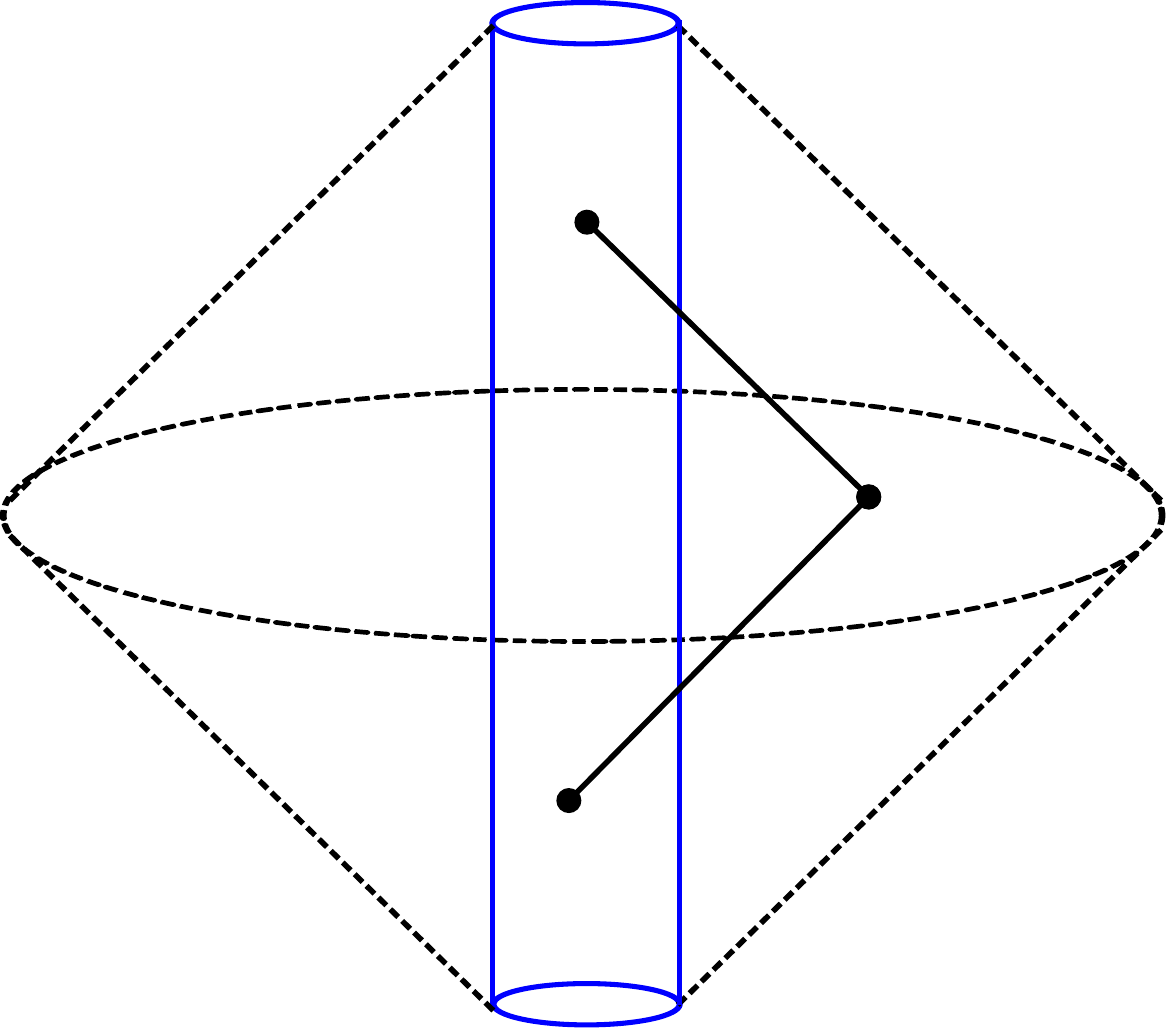
\quad
\def\svgwidth{0.4\textwidth}
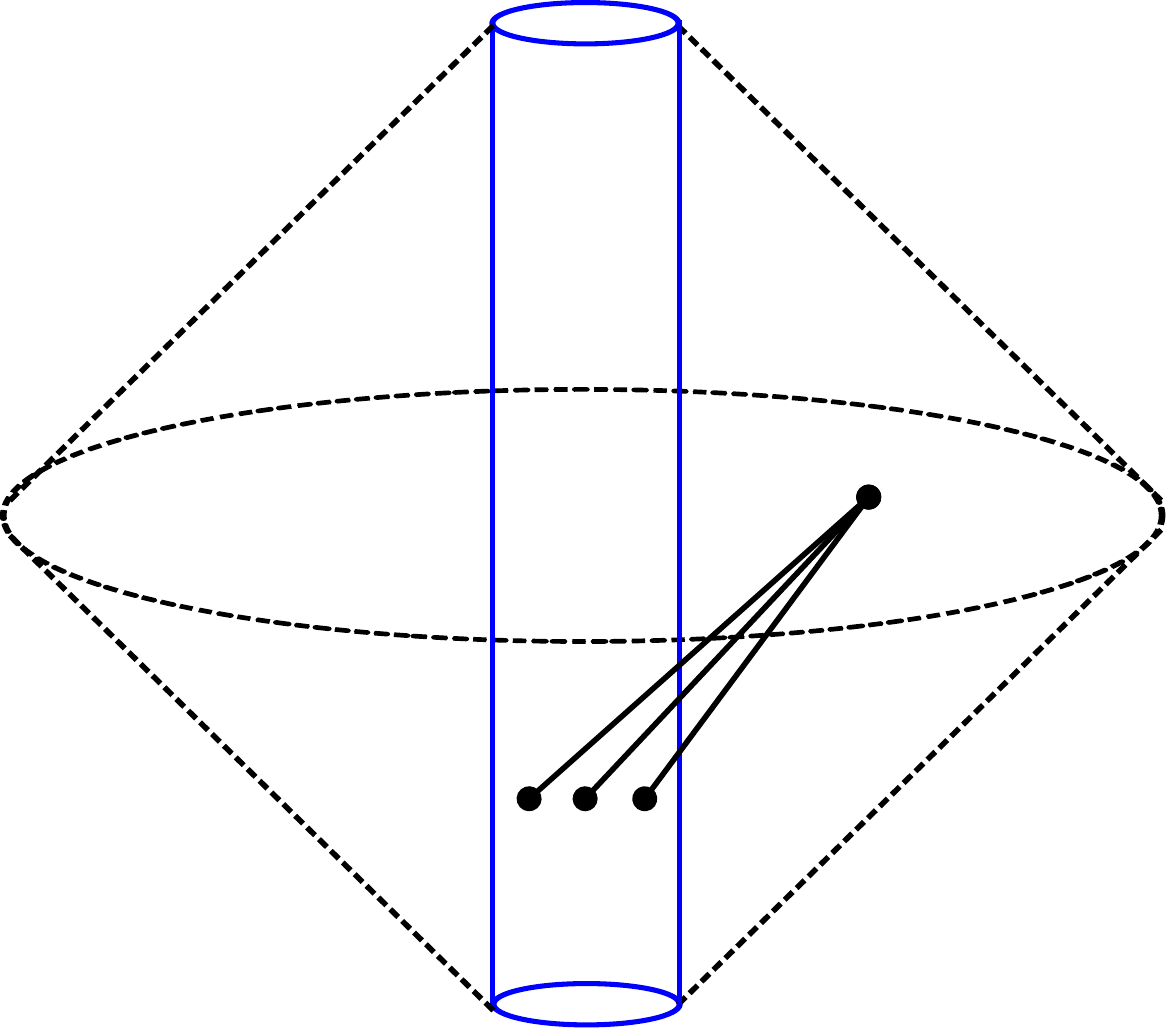
\caption{
Set $\mho$, see (\ref{def_mho}), is the blue cylinder, and set $\mathbb D$, see (\ref{def_diamond}), is the diamond like region drawn with dashed curves. 
{\em Left.} Points $x,z \in \mho$ and $y \in \mathbb D$
as in (\ref{gamma_in_out}). The geodesic segment from $x$ to $y$ is $\gamma_{in}$ and the segment from $y$ to $z$ is $\gamma_{out}$. We write also $\gamma_{in} = \gamma_{y \gets x}$ and $\gamma_{out} = \gamma_{z \gets y}$.
{\em Right.}
Points $x_\mathbf{j}$, $\mathbf{j} = 1, 2, 3$, see (\ref{def_xj}),
together with the geodesic segments $\gamma_{y \gets x_\mathbf{j}}$. Here $x_1 = x$.
}\label{fig_diamond}
\end{figure}

We will use the following shorthand notations related to (\ref{gamma_in_out}),
    \begin{align}\label{def_diamonds_etc}
\mathbb L &= \{(x,y) \in M^2 : \text{there is a lightlike geodesic joining $x$ and $y$}\},
\\\notag
\mathbb S^+(\mho) &= \{(x,y,z) \in M^3 : (x,y), (y,z) \in \mathbb L,\ x < y < z,\ x,z \in \mho,\ y \notin \mho \}.
    \end{align}
Assume from now on that $(M,g)$ is the $1+3$-dimensional Minkowski space, and denote the geodesic joining a pair $(x,y) \in \mathbb L$ by $\gamma_{y \gets x}$. What follows is independent from the parametrization of $\gamma_{y \gets x}$, but let us fix it by requiring that $\gamma_{y \gets x}(0) = x$
and $\dot \gamma_{y \gets x}(0) = (1,\theta)$ where $\theta \in S^2$.
We define also the broken non-abelian light ray transform of the connection $A$ by
    \begin{align}\label{def_S}
\mathbf{S}^A_{z \gets y \gets x} &= \mathbf{P}^A_{z \gets y} \mathbf{P}^A_{y \gets x}, \quad (x,y,z) \in \mathbb S^+(\mho).
    \end{align}

Now we can summarize the above discussion as follows: 
for each $(x,y,z) \in \mathbb S^+(\mho)$ 
we want to choose $f_\mathbf{j}$, $\mathbf{j} = 1, 2, 3$, so that $\phi_{in}$ is singular along $\gamma_{y \gets x}$ and that 
$\phi_{out}$
is singular along $\gamma_{z \gets y}$.
Moreover, we will choose $f_\mathbf{j}$, $\mathbf{j} = 1, 2, 3$, in such a way that, in a suitable microlocal sense, $\phi_{out}$ is a conormal distribution. Then we will show using Proposition~\ref{prop_P_rescaling} that $\mathbf{S}^A_{z \gets y \gets x}$ is determined by the principal symbol of $\phi_{out}|_\mho$ when $f_\mathbf{j}$, $\mathbf{j} = 1, 2, 3$, are varied. 
As each $f_\mathbf{j}$ is supported in $\mho$, 
this will show that the source-to-solution map $L_A$ determines $\mathbf{S}^A_{z \gets y \gets x}$ for all $(x,y,z) \in \mathbb S^+(\mho)$.

\subsection{The linear span of three lightlike vectors}
\label{sec_linalg_lemma}

We will begin with a lemma in linear algebra that will be a key component when proving that singularities propagating along any $\gamma_{z \gets y}$, with the pair $(y,z)$ in 
    \begin{align*}
\mathbb S^{out}(\mho) = \{(y,z) : \text{$(x,y,z) \in \mathbb S^+(\mho)$ for some $x \in \mho$}\},
    \end{align*}
can be generated by choosing suitable $f_\mathbf{j}$, $\mathbf{j} = 1, 2, 3$, supported in $\mho$.
We believe that this lemma can be also used to simplify the proofs of previous results, such as \cite{KLU1, KLU,  LUW}. These results are based on using a superposition of four singular sources, which leads to more complicated computations. 

To highlight the difference, we recall that all the previous results consider non-linear wave equations on $1+3$-dimensional Lorentzian manifolds, and use the fact that if $N^* K_\mathbf{j} \setminus 0$, $\mathbf{j}=1,2,3,4$, are flowouts analogous to $N^* K \setminus 0$ in Example \ref{ex_f_delta}, then generically $\bigcap_{\mathbf{j}=1}^4 K_\mathbf{j}$ is discrete. It follows that $N_y^* (\bigcap_{\mathbf{j}=1}^4 K_\mathbf{j}) = T_y^* M$
for a point $y \in \bigcap_{\mathbf{j}=1}^4 K_\mathbf{j}$. Then, using notation analogous with the above, it is shown that
$\phi_{out} = \p_{\epsilon_1}\p_{\epsilon_2}\p_{\epsilon_3}\p_{\epsilon_4} \phi|_{\epsilon = 0}$ can be made singular along any $\gamma_{z \gets y}$ with $(y,z) \in \mathbb S^{out}(\mho)$. As $\dot \gamma_{z \gets y}(0) \in N_y^* (\bigcap_{\mathbf{j}=1}^4 K_\mathbf{j})$ trivially, there is no geometric obstruction for $\phi_{out}$ being singular on $\dot \gamma_{z \gets y}$ in view of the product calculus for conormal distributions, see Lemma \ref{lem : product of conormal distributions} below. 

On the other hand, we will use only three flowouts, and the intersection $\bigcap_{\mathbf{j}=1}^3 K_\mathbf{j}$ will be of dimension one. Lemma \ref{Lauri's Lemma} implies, however, that for any fixed 
$y \in K_1$ and any fixed lightlike $\eta \in T_y^* M$ 
we can guarantee that $(y,\eta) \in N^* (\bigcap_{\mathbf{j}=1}^3 K_\mathbf{j})$
by choosing $K_2$ and $K_3$ carefully.

\begin{remark}\label{rem_two_not_enough}
Consider the solution $\phi(\epsilon)$ of (\ref{eq_emergent_prin})
with vanishing initial conditions and the source $f(\epsilon) = \epsilon_1 f_1 + \epsilon_2 f_2$, with $f_\mathbf{j}$ a conormal distribution supported in $\mho$ satisfying (\ref{support condition}). Regardless of the degree of non-linearity $N \ge 2$, it is not possible to make $\p_{\epsilon_1}\p_{\epsilon_2} \phi|_{\epsilon = 0}$
singular along arbitrary $\gamma_{z \gets y}$ with $(y,z) \in \mathbb S^{out}(\mho)$.
Indeed, using (\ref{span_of_two}) it can be shown that
    \begin{align*}
\singsupp(\p_{\epsilon_1}\p_{\epsilon_2} \phi|_{\epsilon = 0}) \subset \bigcup_{j=1}^2 \singsupp(\p_{\epsilon_\mathbf{j}} \phi|_{\epsilon = 0}).
    \end{align*}
This explains the lower bound $J \ge 3$ in (\ref{lin_lower_bound}).
\end{remark}

The following lemma is formulated for the Minkowski space but, being of pointwise nature, it generalizes to an arbitrary Lorentzian manifold.

\begin{lemma}\label{Lauri's Lemma}
Let $y$ be a point  in $\mathbb{R}^{1 + 3}$, and let $\xi_1, \eta \in T_y \R^{1+3} \setminus 0$ be lightlike. In any neighbourhood of $\xi_1$ in $T_y \mathbb{R}^{1 + 3}$, there exist two lightlike vectors $\xi_2, \xi_3$ such that $\eta$ is in $\linspan(\xi_1, \xi_2, \xi_3)$.  
\end{lemma}
\begin{proof}
The statement is invariant with respect to non-vanishing rescaling of $\xi_1$ and $\eta$, and we assume without loss of generality that $\xi_1 = (1,\xi_1')$ and $\eta = (1,\eta')$.
Let $\nu \in S^2$ be such that $\xi_1', \eta' \perp \nu$. After a rotation in $\mathbb{R}^3$, we may assume that $\nu = (0, 0, 1)$. Then $\xi_1' = (\theta_1, 0)$ and $\eta' = (\theta_0, 0)$ with $\theta_1, \theta_0 \in S^1$. After a rotation in $\mathbb{R}^2$, we may assume that $\theta_1 = (1,0)$. We write $a(r) = \sqrt{1-r^2}$ for $r \in [-1,1]$. Then there is $r_0 \in [-1,1]$, and a choice of sign, such that $\theta_0 = (\pm a(r_0), r_0)$.
To summarize, we may assume without loss of generality that 
    \begin{align}\label{xi_eta_form}
\xi_1 = (1,1,0,0), \quad \eta = (1,\pm a(r_0), r_0, 0).
    \end{align}

It is enough to show that in any neighbourhood $U_0 \subset S^1$ of $\theta_1$ there are $\theta_2, \theta_3 \in U_0$ such that
$(1,\theta_\mathbf{j})$, $\mathbf{j}=1,2,3$, span $\mathbb{R}^3$.
We set $\theta_2 = (a(r),r)$ and $\theta_3 = (a(r),-r)$
where $r > 0$ is chosen to be small enough so that
$\theta_\mathbf{j} \in U_0$, $\mathbf{j}=2,3$.
Now $(1,\theta_\mathbf{j})$, $\mathbf{j}=1,2,3$, are linearly independent since
$$
\begin{vmatrix}
1 & 1 & 1
\\
1 & a(r) & a(r)
\\
0 & r & -r
\end{vmatrix}
= 2 r (1-a(r)) \ne 0.
$$
To summarize, we may choose
    \begin{align}\label{xi23}
\xi_2 = (1,a(r),r,0), \quad \xi_3 = (1,a(r),-r,0).
    \end{align}
\end{proof}

We will need also an explicit version of the above lemma. 
As in the proof, we assume without loss of generality that (\ref{xi_eta_form}) holds, and choose $\xi_2$ and $\xi_3$ according to (\ref{xi23}). Then
    \begin{align}\label{eta_lin_comb}
r^{2}\eta = (-2 b + \mathcal O(r)) \xi_1 + (b + \mathcal O(r)) \xi_2
+ (b + \mathcal O(r)) \xi_3,
    \end{align}
where $b = 1 \mp a(r_0) = 1 \mp \sqrt{1-r_0^2}$.

\begin{remark}\label{rem_signs_mho}
In the context of (\ref{gamma_in_out}), we will take $\xi_1 = \dot{\gamma}_{in}(s_{in})$ and $\eta = \dot{\gamma}_{out}(0)$.
Assuming that $y \notin \mho$,
the condition $\gamma_{out}(s_{out}) \in \mho$ implies that the second component of $\eta$ in 
(\ref{xi_eta_form}), that is, $\pm a(r_0)$, must have negative sign. 
Then $b \ge 1$, and in particular, $\eta$ is not in $\linspan(\xi_j, \xi_k)$ for any $j,k \in \{1,2,3\}$.
This fact will be important in what follows, and the set $\mho$ was chosen to be of the particular form (\ref{def_mho}) to avoid technicalities related, for example, to geodesics that return to $\mho$ after exiting it. 
\end{remark}

\subsection{Three-fold linearization of cubic connection wave equations}

Let us consider a source $f(\epsilon)$ of the form (\ref{def_f}),
and the half-density $u(\epsilon) = |g|^{1/4} \phi(\epsilon)$
where 
$\phi(\epsilon)$ denotes again the solution of (\ref{wave_nonlin}) with the source $f=f(\epsilon)$.
Recalling the definition of $Q$, see (\ref{def_Q}), we
observe that $u=u(\epsilon)$ satisfies the wave equation 
    \begin{align}\label{eqn : cubic wave eqn}
Q(u) = |g|^{1/4} f.
    \end{align}

To simplify the notation, we assume that $\kappa = -1$ in (\ref{eq:model}). The general case is analogous. 
We introduce 
$$
u_{\mathbf{j_1}, \dots, \mathbf{j_l}} = \partial_{\epsilon_{\mathbf{j_1}}} \cdots \partial_{\epsilon_{\mathbf{j_l}}} u, \quad v_{\mathbf{j_1}, \cdots, \mathbf{j_l}} = u_{\mathbf{j_1}, \cdots, \mathbf{j_l}}|_{\epsilon = 0},
$$ 
and write $\Re$ for the real part of a complex number. 
With these notations, we differentiate \eqref{eqn : cubic wave eqn}, and obtain
\begin{equation*}
P u_\mathbf{j} - \Re(\langle u_\mathbf{j}, u\rangle) u - \frac{1}{2} |u|^2 u_\mathbf{j} = |g|^{1/4} f_\mathbf{j}.
\end{equation*} 
Together with the fact that $u|_{\epsilon = 0} = 0$, this implies $P v_\mathbf{j} = |g|^{1/4} f_\mathbf{j}$. Differentiating the equation \eqref{eqn : cubic wave eqn} twice yields \begin{equation}\label{two-fold_lin}
P u_{\mathbf{j}\mathbf{k}} - \Re(\langle u_{\mathbf{j}\mathbf{k}}, u\rangle + \langle u_{\mathbf{j}}, u_\mathbf{k}\rangle) u - \Re(\langle u_{\mathbf{j}}, u\rangle) u_\mathbf{k} - \Re(\langle u_{\mathbf{k}}, u\rangle) u_\mathbf{j} - \frac{1}{2} |u|^2 u_{\mathbf{j}\mathbf{k}} = 0,
\end{equation} 
and $v_{\mathbf{j}\mathbf{k}} = 0$.
Finally, one more differentiation gives the desired linear equation, that we call the three-fold linearization,
\begin{equation}\label{eqn : linearized wave equation }P v_{123} = \Re(\langle v_1, v_2\rangle) v_3 + \Re(\langle v_1, v_3\rangle) v_2 + \Re(\langle v_2, v_3\rangle) v_1.\end{equation}
The right-hand side of \eqref{eqn : linearized wave equation } is the sum of products of $v_1, v_2, v_3$, and we call it the three-fold interaction of these three solutions to the linear wave equation (\ref{eq_wave_lin}).

\begin{remark}\label{rem_lin_N}
Consider the solution $\phi(\epsilon)$ of (\ref{eq_emergent_prin})
with vanishing initial conditions and the source $f(\epsilon) = \sum_{j=1}^N \epsilon_\mathbf{j} f_\mathbf{j}$, with $f_\mathbf{j}$ a distribution supported in $\mho$. 
Analogously to (\ref{two-fold_lin}), we see that $\p_{\epsilon}^\alpha \phi|_{\epsilon = 0} = 0$
for any multi-index $\alpha \in \N^{1+3}$
satisfying $|\alpha| < N$ where $N$ is the degree of non-linearity in the equation (\ref{eq_emergent_prin}). 
This explains the lower bound $J \ge N$ in (\ref{lin_lower_bound}).
\end{remark}

\subsection{Wave interactions as products of conormal distributions}

Consider a triple $(x,y,z) \in \mathbb S^+(\mho)$ and let us construct $f_\mathbf{j}$, $\mathbf{j} = 1,2,3$, such that $\phi_{out}$, defined by (\ref{phi_in_out}), 
is singular on $\gamma_{z \gets y}$.
We write 
    \begin{align}\label{def_eta_xi}
\eta = \dot \gamma_{z \gets y}^*(0), \quad
\xi_1 = \dot \gamma_{y \gets x}^*(s_{in}),
    \end{align}
where $s_{in} > 0$ satisfies $\gamma_{y \gets x}(s_{in}) = y$. 

The geodesic $\gamma$ with the initial condition $(x, \xi)$
is denoted by $\gamma(\cdot;x,\xi)$.
Let $\mathcal V \subset T_y^* M$ be a small enough neighbourhood of $\xi_1$ so that for all $\xi \in \mathcal V$ it holds that $\gamma(-s_{in};y,\xi) \in \mho$ and that $\xi$ is future pointing. Let $\xi_2, \xi_3 \in \mathcal V$ be two lightlike vectors as in Lemma \ref{Lauri's Lemma}
and write 
    \begin{align}\label{def_xj}
x_\mathbf{j} = \gamma(-s_{in};y,\xi_\mathbf{j}), \quad \mathbf{j}=1,2,3.
    \end{align}
With a slight abuse of notation, we write also 
$\xi_\mathbf{j} = \dot \gamma(-s_{in};y,\xi_\mathbf{j})$, see Figure \ref{fig_diamond}, right, for the geometric setup. 

Analogously to Example \ref{ex_f_delta}, we let 
$c_\mathbf{j} \in E_{x_\mathbf{j}} \setminus 0$, and let $\chi_\mathbf{j}$ be a pseudodifferential operator 
with the following properties:
\begin{itemize}
\item[($\chi$1)] $\sigma[\chi_\mathbf{j}]$ is positively homogeneous of degree $q$, symmetric with respect to $0 \in T^* M$, and real valued,
\item[($\chi$2)] $\sigma[\chi_\mathbf{j}] \ne 0$ near $(x_\mathbf{j}, \xi_\mathbf{j})$,
\item[($\chi$3)] $\WF(\chi_\mathbf{j})$ is contained in a small neighbourhood of $\ccl(x_\mathbf{j}, \pm \xi_\mathbf{j})$.
\end{itemize}

The degree $q$ of $\chi_{\mathbf{j}}$ is chosen to be small enough so that $f_\mathbf{j} \in C^4(M)$
where 
    \begin{align}\label{def_fj}
f_\mathbf{j} \in I(M; N^* \{x_\mathbf{j}\} \setminus 0; E), \quad f_\mathbf{j} = c_\mathbf{j} \chi_\mathbf{j} \delta_{x_\mathbf{j}}.
    \end{align}
Here $\delta_{x_\mathbf{j}}$ is the Dirac delta distribution at $x_\mathbf{j}$.
Moreover, we choose $\chi_{\mathbf{j}}$ so that $\supp(\chi_{\mathbf{j}}\delta_{\mathbf{j}}) \subset \mho$ and that the support condition (\ref{support condition}) is satisfied.
Recalling that $\mathcal C \subset C_0^4(\mho; E)$ is the domain of the source-to-solution map $L_A$, see (\ref{def_L}),
we have then that the linear combination $\epsilon_1 f_1 + \epsilon_2 f_2  + \epsilon_3 f_3$ is in $\mathcal C$ for small enough $\epsilon_\mathbf{j} \ge 0$.

Recall that $v_\mathbf{j}$ denotes the solution of 
    \begin{align}\label{eq_for_vj}
P v_\mathbf{j} = |g|^{1/4} f_\mathbf{j}, \quad 
v_\mathbf{j}|_{t < 0} = 0.
    \end{align}
Moreover, we denote the restriction of $v_\mathbf{j}$ on $\R^{1+3} \setminus \{x_\mathbf{j}\}$ still by $v_\mathbf{j}$, and write $K_\mathbf{j} = x_\mathbf{j} + K$ where $K$ is as in (\ref{K_delta}). That is, writing $x_\mathbf{j} = (t_\mathbf{j},x_\mathbf{j}')$,
    \begin{align}\label{def_Kj}
K_\mathbf{j} = \{(t_\mathbf{j} + s, x_\mathbf{j}' + s \theta) \in \R^{1+3}: s > 0,\ \theta \in S^2\}.
    \end{align}
Then $v_\mathbf{j} \in I(\R^{1+3} \setminus x_\mathbf{j}; N^* K_\mathbf{j} \setminus 0; E \otimes \Omega^{1/2})$. 
As in Example \ref{ex_f_delta}, property ($\chi$3) implies that $\singsupp(v_\mathbf{j})$ is contained in a small neighbourhood $\mathscr S_\mathbf{j} \subset \overline{K_\mathbf{j}}$ of $\gamma_{y \gets x_\mathbf{j}}(\R_+)$.

Clearly $y \in \bigcap_{\mathbf{j} = 1}^3 \gamma_{y \gets x_\mathbf{j}}(\R_+) \subset \bigcap_{\mathbf{j} = 1}^3 K_\mathbf{j}$.
It follows from Remark \ref{rem_signs_mho} that the covectors $\xi_{\mathbf{j}}$, $\mathbf{j} = 1,2,3$, are linearly independent. This again implies that 
    \begin{align*}
N_y^* (K_1 \cap K_2 \cap K_3)
= \linspan(\xi_1, \xi_2, \xi_3).
    \end{align*}
When $\WF(\chi_\mathbf{j})$, $\mathbf{j} = 1,2,3$, are  small enough, this also guarantees for distinct $\mathbf{j}$, $\mathbf{k}$, $\mathbf{l}$ that $K_\mathbf{j}$ and $K_\mathbf{k}$, as well as $K_\mathbf{j}$ and $K_\mathbf{k} \cap K_\mathbf{l}$, are transversal in $\mathscr S := \mathscr S_1 \cup \mathscr S_2 \cup \mathscr S_3$.

Let us now consider products of the form $\Re(\langle v_\mathbf{j}, v_\mathbf{k} \rangle) v_\mathbf{l}$, with distinct $\mathbf{j}, \mathbf{k}, \mathbf{l}$, that appear on the right-hand side of (\ref{eqn : linearized wave equation }). As $s_{in} > 0$, it holds that $y \ne x_\mathbf{j}$ for each $\mathbf{j} = 1,2,3$, and $v_\mathbf{j}$ are conormal distributions near $y$. Thus their products can be analysed using the product calculus for conormal distributions. We recall this calculus in the next lemma, that is a variant of \cite[Lemma 1.1]{Greenleaf-Uhlmann-CMP-1993}. With obvious modifications it holds also when the conormal distributions $u_\mathbf{j}$, $\mathbf{j} = 1,2$, take values on $E \otimes \Omega^{1/2}$ and the product is defined in terms of $\pair{\cdot, \cdot}_E$, and also when $u_1$ takes values on $\Omega^{1/2}$ and $u_2$ on $E \otimes \Omega^{1/2}$ and the product is defined in terms of the scalar-vector product on $E$. 

\begin{lemma}\label{lem : product of conormal distributions}
Let $K_\mathbf{j} \subset M$, $\mathbf{j} = 1, 2$, be transversal submanifolds, and let $u_\mathbf{j}$ be a conormal distribution in $I(M; N^\ast K_\mathbf{j} \setminus 0; \Omega^{1/2})$. For a fixed nowhere vanishing half-density $\mu \in C^\infty(M; \Omega^{1/2})$, we define the product of $u_1$ and $u_2$ by
$$
u_1 u_2 = \mu ((\mu^{-1} u_1)(\mu^{-1} u_2)).
$$   
Let $\chi$ be a pseudodifferential operator with $\WF(\chi)$ disjoint from both $N^\ast K_\mathbf{j}$, $\mathbf{j} = 1,2$. 
It follows that 
$$
\chi (u_1 u_2) \in I(M; N^\ast (K_1 \cap K_2) \setminus 0; \Omega^{1/2}),
$$ 
with the principal symbol, ignoring the $2 \pi$ and $\imath$ factors, 
$$
\sigma[\chi(u_1 u_2)](x,\xi) = \mu^{-1}(x) \sigma[\chi](x,\xi) \sigma[u_1](x,\xi^{(1)})\sigma[u_2](x,\xi^{(2)}).
$$ 
where $\xi = \xi^{(1)} + \xi^{(2)}$ with $\xi^{(\mathbf{j})} \in N_x^\ast K_\mathbf{j} \setminus 0$ and $x \in K_1 \cap K_2$.
\end{lemma}

Let us remark that, due to the transversality of $K_1$ and $K_2$, it holds that 
$$
N_x^* (K_1 \cap K_2) = N_x^* K_1 \oplus N_x^* K_2,
\quad x \in K_1 \cap K_2.
$$ 
Therefore the decomposition $\xi = \xi^{(1)} + \xi^{(2)}$, $\xi^{(\mathbf{j})} \in N_x^\ast K_\mathbf{j}$, is unique for any covector $\xi \in N_x^* (K_1 \cap K_2)$. The condition that both $\xi^{(1)}$ and $\xi^{(2)}$ are non-zero is equivalent with
    \begin{align*}
(x,\xi) \in 
N^\ast (K_1 \cap K_2) \setminus (N^\ast K_1 \cup N^\ast K_2).
    \end{align*}

We return to our study of $v_\mathbf{j}$, $\mathbf{j} = 1,2,3$.
Due to Remark \ref{rem_signs_mho}, we can choose a pseudodifferential operator $\chi$ such that $\chi = 1$ near $\ccl(y,\eta)$, and that $\WF(\chi)$ is contained in a small conical neighbourhood of $\ccl(y,\eta)$.
Then applying Lemma \ref{lem : product of conormal distributions} twice, we obtain
    \begin{align}\label{conormal_prod}
\chi(\Re(\langle v_\mathbf{j}, v_\mathbf{k} \rangle) v_\mathbf{l}) \in I(\R^{1+3}; \Lambda_0; E \otimes \Omega^{1/2}),
    \end{align}
where $\Lambda_0 = N^\ast (K_1 \cap K_2 \cap K_3) \setminus 0$ and $\mathbf{j}, \mathbf{k}, \mathbf{l} \in \{1, 2, 3\}$ are distinct. 
For the convenience of the reader, we give a detailed proof of this. 
\begin{proof}[Proof of (\ref{conormal_prod})]
To simplify the notation, we will consider the product
$\Re(\langle v_1, v_2 \rangle) v_3$, the other cases being analogous. Write
     \begin{align}\label{eta_decomp}
\eta = \eta^{(1)} + \eta^{(2)} + \eta^{(3)}, \quad
\eta^{(\mathbf{j})} \in N_y^*K_\mathbf{j} =  \linspan(\xi_\mathbf{j}),
    \end{align}
and let $\chi_0$ and $\chi_1$ be pseudodiffential operators such that 
    \begin{align*}
\WF(\chi_0) \cap (N^* K_1 \cup N^* K_2) = \emptyset,
\quad 
\WF(\chi_1) \cap (N^* (K_1 \cap K_2) \cup N^* K_3) = \emptyset,
    \end{align*}
$\chi_0 = 1$ near $\ccl(y,\eta^{(1)} + \eta^{(2)})$
and $\chi_1 = 1$ near $\ccl(y,\eta)$.
Recall that $\eta$ is not in $\linspan(\xi_j, \xi_k)$ for any $j,k \in \{1,2,3\}$, see Remark \ref{rem_signs_mho}, and this guarantees that $\chi_0$ and $\chi_1$ with the above properties exist. Furthermore, as $s_{in} > 0$, we can choose $\chi_0$ so that $\WF(\chi_0)$ is also disjoint from $\{x_1, x_2, x_3\}$.
Now Lemma \ref{lem : product of conormal distributions} implies that 
    \begin{align}\notag
\chi_0(\Re\langle v_1, v_2 \rangle) &\in I(\R^{1+3}; N^\ast (K_1 \cap K_2) \setminus 0; \Omega^{1/2}),
\\\label{double_cutoff_lem2}
\chi_1(\chi_0(\Re\langle v_1, v_2 \rangle) v_3) &\in I(\R^{1+3}; \Lambda_0; E \otimes \Omega^{1/2}).
    \end{align}

We write $w = (1-\chi_0)(\Re\langle v_1, v_2 \rangle) v_3$ and will show that $(y,\eta) \notin \WF(w)$. As $\WF(w)$ is conical and closed, it then follows that $\WF(\chi_1) \cap \WF(w) = \emptyset$
whenever $\WF(\chi_1)$ is contained in a small enough conical neighbourhood of $\ccl(y,\eta)$. In this case $\chi_1 w$ is smooth, and (\ref{double_cutoff_lem2}), together with
    \begin{align*}
\chi_1(\Re(\langle v_1, v_2 \rangle) v_3)
= \chi_1(\chi_0(\Re\langle v_1, v_2 \rangle) v_3) + \chi_1 w,
    \end{align*}
implies that $\chi_1(\Re(\langle v_1, v_2 \rangle) v_3) \in I(\R^{1+3}; \Lambda_0; E \otimes \Omega^{1/2})$.

To simplify the notation, we write $\WF_y(w) = \WF(w) \cap T_y^* \R^{1+3}$.
It remains to show that $\eta \notin \WF_y(w)$.
We have $\WF_y(v_{\mathbf{j}}) \subset N_y^*K_\mathbf{j} =  \linspan(\xi_\mathbf{j})$. As the vectors $\xi_1$ and $\xi_2$ are linearly independent, \cite[Th. 8.2.10]{Hormander1} implies that 
    \begin{align*}
\WF_y(\pair{v_1,v_2}) \subset N^*_y (K_1 \cap K_2)
= \linspan(\xi_1,\xi_2).
    \end{align*}
Moreover, as $\chi_0 = 1$ near $\ccl(y,\eta^{(1)} + \eta^{(2)})$, 
    \begin{align*}
\WF_y((1-\chi_0)(\Re\langle v_1, v_2 \rangle)) \subset \linspan(\xi_1,\xi_2) \setminus \linspan(\eta^{(1)} + \eta^{(2)}).
    \end{align*}
Using again \cite[Th. 8.2.10]{Hormander1}, we have
    \begin{align*}
\WF_y(w) \subset \{ \zeta^1 + \zeta^2 : \zeta^1 \in (\linspan(\xi_1,\xi_2) \setminus \linspan(\eta^{(1)} + \eta^{(2)})) \cup \{0\},\ \zeta^2 \in \linspan(\xi_3) \}.
    \end{align*}
As the direct sum decomposition (\ref{eta_decomp}) is unique, and as $\eta^{(1)} + \eta^{(2)} \ne 0$ by Remark \ref{rem_signs_mho}, it holds that $\eta \notin \WF_y(w)$.
\end{proof}

Let us denote the right-hand side of (\ref{eqn : linearized wave equation }) by $f_{out}$, that is,
    \begin{align*}
f_{out} = \Re(\langle v_1, v_2\rangle) v_3 + \Re(\langle v_1, v_3\rangle) v_2 + \Re(\langle v_2, v_3\rangle) v_1.
    \end{align*}
Furthermore, we write $v_{out}$ and $w$ for the solutions of the following two wave equations
    \begin{align}\label{eq_wave_out}
\left\{ \begin{array}{l}
Pv_{out} = \chi f_{out},\\ v_{out}|_{t < 0} = 0,
\end{array} \right. \qquad \left\{ \begin{array}{l}
Pw = (1 - \chi) f_{out},\\ w|_{t < 0} = 0,
\end{array} \right.
    \end{align}
where $\chi$ is as in (\ref{conormal_prod}).
Then the solution $v_{123}$ of (\ref{eqn : linearized wave equation }) satisfies $v_{123} = v_{out} + w$,
and due to (\ref{conormal_prod}), it holds that $\chi f_{out} \in  I(\R^{1+3}; \Lambda_0; E \otimes \Omega^{1/2})$. We will treat $w$ as a remainder term.

Recall that $y = \gamma_{z \gets y}(0)$ and $\eta = \dot \gamma_{z \gets y}^*(0)$. 
As $\chi = 1$ near $\ccl(y,\eta)$, it holds that 
$$
(y,\eta) \notin \WF((1-\chi) f_{out}).
$$
As $\singsupp(f_{out}) \subset \mathscr S$ and as $\mathscr S$ is a small neighbourhood of $\bigcup_{\mathbf{j}=1}^3 \gamma_{y \gets x_\mathbf{j}}(\R_+)$
in $\bigcup_{\mathbf{j}=1}^3 \overline{K_\mathbf{j}}$,
we see that the bicharacteristic 
    \begin{align}\label{def_beta_out}
\beta_{out}(s) = (\gamma_{z \gets y}(s), \gamma_{z \gets y}^*(s))
    \end{align}
does not intersect $\WF((1-\chi) f_{out})$.
Writing 
    \begin{align*}
(z,\zeta) = \beta_{out}(s_{out})
    \end{align*}
for some $s_{out} > 0$ and $\zeta \in T_z^* M \setminus 0$, it follows from \cite[Th. 23.2.9]{Hormander-Vol3} that $(z,\zeta) \notin \WF(w)$. 

Let us now consider the future flowout $\Lambda_1$ from 
$\Lambda_0 \cap \Sigma(P)$, with $\Lambda_0$ as in (\ref{conormal_prod}). We will show in Appendix \ref{appendix3} that $\Lambda_1 \setminus \Lambda_0$ coincides with a conormal bundle, and thus we can apply Theorem \ref{th_parametrix} to the equation for $v_{out}$ in (\ref{eq_wave_out}). 
Let us point out that this study of the structure of $\Lambda_1$ is not essential for the proof. We could alternatively treat $v_{out}$ as a Lagrangian distribution, but then its symbol should be viewed as a section of $E \otimes \Omega^{1/2} \otimes \mathscr{M}$, with $\mathscr{M}$ the Maslov bundle over $\Lambda_1$. However, we prefer to avoid technicalities related to general Lagrangian distributions in the present paper. 

Theorem \ref{th_parametrix} implies that $v_{out} \in I(M; \Lambda_0, \Lambda_1; E\otimes \Omega^{1/2} )$,
and as $(z,\zeta) \notin \WF(w)$, we may choose a pseudodifferential operator $\tilde \chi$ acting on $\mathcal D'(\mho)$
such that $\sigma[\tilde \chi] = 1$ near $(z,\zeta)$
and that 
 $\tilde \chi (v_{123}|_\mho) \in I(\mho; \Lambda_1; E\otimes \Omega^{1/2} )$. Then $\sigma[\tilde \chi (v_{123}|_\mho)](z,\zeta) =\sigma[v_{out}](z,\zeta)$
 and in terms of the source-to-solution map this reads as
     \begin{align}\label{recovery_of_prinsymb}
\sigma[\tilde \chi |g|^{1/4}\p_{\epsilon_1}\p_{\epsilon_2}\p_{\epsilon_3} L_A ( \epsilon_1 f_1 + \epsilon_2 f_2 + \epsilon_3 f_3 )|_{\epsilon = 0} ](z,\zeta) 
=\sigma[v_{out}](z,\zeta).
    \end{align}
This shows that $L_A$ determines $\sigma[v_{out}](z,\zeta)$.
We will recover $\mathbf{S}^A_{z \gets y \gets x}$ by considering $\sigma[v_{out}]$ at $(z,\zeta) \in \Lambda_1$.

\subsection{From the source-to-solution map to the broken light ray transform}

Having understood the propagation and interaction of the linear waves, we now prove that the source-to-solution map determines  the broken non-abelian light ray transform.

\begin{theorem}\label{th_from_L_to_S}
Let $A$ and $B$ be two connections in $\R^{1+3}$ such that $L_{A}=L_{B}$, as in Theorem \ref{thm:main}.
Then $\mathbf{S}^A_{z \gets y \gets x} = \mathbf{S}^B_{z \gets y \gets x}$
for all $(x,y,z) \in \mathbb S^+(\mho)$.
\end{theorem}

We give a constructive proof of Theorem \ref{th_from_L_to_S} in the form of a method that recovers $\mathbf{S}^A_{z \gets y \gets x}$ for any $(x,y,z) \in \mathbb S^+(\mho)$ given $L_A$.
Therefore we continue considering only a single connection $A$.

Let $(x,y,z) \in \mathbb S^+(\mho)$ and let $\eta, \xi_1 \in T_y^* M$ be defined by (\ref{def_eta_xi}). Without loss of generality we may assume that $\eta$ and $\xi_1$  are of the form (\ref{xi_eta_form}), and letting again $\xi_2, \xi_3 \in T_y^* M$ be as in Lemma \ref{Lauri's Lemma}, we have (\ref{eta_lin_comb}). 
Recall that $K_\mathbf{j}$ is defined by (\ref{def_Kj}). Decomposing $\eta$ as in (\ref{eta_decomp}),
it holds for small $r > 0$ that 
    \begin{align}\label{eta_components}
\eta^{(1)} = -r^{-2} \lambda_1 \xi_1, \quad \eta^{(\mathbf{j})} = r^{-2} \lambda_\mathbf{j} \xi_\mathbf{j}, \quad \mathbf{j} = 2,3, 
    \end{align}
where
    \begin{align*}
\lambda_1 = 2 b + \mathcal O(r),
\quad
\lambda_\mathbf{j} = b + \mathcal O(r),
\quad \mathbf{j} = 2,3.
    \end{align*}
We recall also that $b \ge 1$, see Remark \ref{rem_signs_mho}.
In particular, $\lambda_\mathbf{j} > 0$, $\mathbf{j} = 1,2,3$, for small $r > 0$.
Let $v_\mathbf{j}$ be again the solution of (\ref{eq_for_vj}) with $f_\mathbf{j}$ as in (\ref{def_fj}), and write 
    \begin{align*}
\hat{v}_\mathbf{j} = \sigma[v_\mathbf{j}](y,\eta^{(\mathbf{j})}), \quad \mathbf{j} = 1,2,3.
    \end{align*}
The product calculus in Lemma \ref{lem : product of conormal distributions} implies that $\chi f_{out}$  in (\ref{eq_wave_out}) satisfies
    \begin{align}\label{sigma_F}
|g(y)|^{1/2} \sigma[\chi f_{out}](y,\eta) = \Re(\langle \hat{v}_1, \hat{v}_2\rangle) \hat{v}_3 + \Re(\langle \hat{v}_1, \hat{v}_3\rangle) \hat{v}_2 + \Re(\langle \hat{v}_2, \hat{v}_3\rangle) \hat{v}_1,
    \end{align}
where we used $\chi(y,\eta) = 1$.

We will apply Proposition \ref{prop_P_rescaling} to  (\ref{eq_for_vj}). 
Toward that end, we choose a strictly positive $\omega_\mathbf{j} \in C^\infty(\overline{N^* K_\mathbf{j}}; \Omega^{1/2})$ satisfying (\ref{omega_symmetry}),
and denote by $\rho_\mathbf{j}$ the factor (\ref{def_rho}) along the bicharacteristic $\beta_\mathbf{j}(s) = (\gamma_{y \gets x_\mathbf{j}}(s), \dot \gamma_{y \gets x_\mathbf{j}}^*(s))$.
It follows from (\ref{def_fj}) and ($\chi$1) that in any coordinates there holds, ignoring $2 \pi$ factors,
    \begin{align}\label{from_delta_to_f}
\sigma[|g|^{1/4} f_\mathbf{j}](x_\mathbf{j}, \pm \lambda \xi_\mathbf{j})
&= c_\mathbf{j} |g(x_\mathbf{j})|^{1/4} \sigma[\chi_\mathbf{j}](x_\mathbf{j},\pm \lambda \xi_\mathbf{j})
= \lambda^q c_\mathbf{j} \alpha_\mathbf{j}, \quad \lambda > 0,
    \end{align}
where $\alpha_\mathbf{j} = |g(x_\mathbf{j})|^{1/4} \sigma[\chi_\mathbf{j}](x_\mathbf{j}, \xi_\mathbf{j})$.
Of course, $|g(x_\mathbf{j})|^{1/4} = 1$ in the Cartesian coordinates, but $\alpha_\mathbf{j}$ transforms as half-density. Observe, however, that $\alpha_\mathbf{j} > 0$ in any coordinate system. 

Let us translate the origin to $x_1$ and pass to the coordinates $(\tilde x^0, \tilde x', \tilde \xi_0, \tilde \xi')$ in (\ref{Fm_coords}) in Appendix \ref{appendix2} below.
In these coordinates $(x_1, -\lambda\xi_1) = (0; 0, -\lambda\xi_1')$ for $\lambda > 0$. Moreover, (\ref{transp_init_m}) and (\ref{from_delta_to_f}) give
$$
\sigma[v_1](0, -\lambda \xi_1') = - \iota_0 \lambda^{q-1} c_1 \alpha_1, \quad \lambda > 0,
$$ 
where $\iota_0 \in \C \setminus 0$ is a constant that contains all the $2 \pi$ and $\imath$ factors. 
In particular, the positive homogeneity (\ref{poshomog_at_x}) holds for $v_1$ at $(x_1, -\xi_1)$.
The coordinates (\ref{Fm_coords}) can be used in a neighbourhood of the bicharacteristic segment $\beta_1^-(s) = (\gamma_{y \gets x_1}(s), -\dot \gamma_{y \gets x_1}^*(s))$, $s \in [0,s_{in}]$, and 
in these coordinates $(y, -\xi_1) = (s_{in} 0; 0, -\xi_1')$.
It follows from Proposition \ref{prop_P_rescaling} and (\ref{eta_components}) that 
    \begin{align*}
e^{\rho_1(s_{in})}(\omega_1^{-1}\sigma[v_1])(s_{in}, -r^{-2} \lambda_1 \xi_1') 
= (r^{-2} \lambda_1)^{q-1} \mathbf{P}^A_{y \gets {x_1}} ( (\omega_1^{-1}\sigma[v_1])(0, -\xi_1') ),
    \end{align*}
and therefore in the coordinates (\ref{Fm_coords}),
    \begin{align}\label{hat_v1}
\hat v_1 = -(r^{-2} \lambda_1)^{q-1} \iota_0 \tilde \alpha_1 \mathbf{P}^A_{y \gets {x_1}} c_1,
    \end{align}
where $\tilde \alpha_1 := e^{-\rho_1(s_{in})}
\omega_1(s_{in}, \xi_1') \omega_1^{-1}(0,\xi_1') \alpha_1$ contains all the volume factors.

Analogously, we may use coordinates of the form (\ref{Fp_coords}) in a neighbourhood of the bicharacteristic segment $\beta_{\mathbf j}(s)$, $s \in [0,s_{in}]$, $\mathbf j = 2,3$, to obtain
    \begin{align}\label{hat_v23}
\hat v_{\mathbf j} = (r^{-2} \lambda_{\mathbf j})^{q-1} \iota_0 \tilde \alpha_{\mathbf j} \mathbf{P}^A_{y \gets {x_{\mathbf j}}} c_{\mathbf j}, \quad \mathbf j = 2,3,
    \end{align}
where $\tilde \alpha_{\mathbf j} = e^{-\rho_{\mathbf j}(s_{in})}
\omega_{\mathbf j}(s_{in}, \xi_{\mathbf j}') \omega_{\mathbf j}^{-1}(0,\xi_{\mathbf j}') \alpha_{\mathbf j}$.
We have given the volume factors $\tilde \alpha_{\mathbf j}$, $\mathbf j = 1,2,3$, in different coordinate systems, however, the way they transform under changes of coordinates is inconsequential. For our present purposes it is enough to observe that they are independent from the connection $A$ and that they satisfy $\tilde \alpha_{\mathbf j} > 0$ in any coordinate system. Let us also emphasize that $\hat v_{\mathbf j}$, as the value of a section of $E \otimes \Omega^{1/2}$ over $K_{\mathbf j}$ at the point 
$(y,\eta^{(\mathbf{j})})$, is a coordinate invariant quantity. 

Combining (\ref{sigma_F}), (\ref{hat_v1}) and (\ref{hat_v23}) yields $|g(y)|^{1/2} \sigma[\chi f_{out}](y,\eta) = \alpha_{in} c_{in}$ where 
    \begin{align*}
c_{in} &= 
\Re(\langle c_{1,in} , c_{2,in} \rangle) 
c_{3,in} 
+ 
\Re(\langle c_{1,in}, c_{3,in} \rangle)
c_{2,in} 
+ 
\Re(\langle c_{2,in}, c_{3,in} \rangle) 
c_{1,in}),
\\
\alpha_{in} &=
- r^{-6(q-1)} (\lambda_1\lambda_2\lambda_3)^{q-1} |\iota_0|^2 \iota_0 \tilde \alpha_1\tilde \alpha_2\tilde \alpha_3, \quad c_{\mathbf j,in} = \mathbf{P}^A_{y \gets x_{\mathbf j}} c_{\mathbf j}.
    \end{align*}
Similarly with the above, we may apply Proposition \ref{prop_P_rescaling} to (\ref{eq_wave_out}) and the bicharacteristic $\beta_{out}$ defined by (\ref{def_beta_out}). This gives, ignoring $\imath$ and $2\pi$ factors,
    \begin{align*}
\sigma[v_{out}](z,\zeta) = \alpha_{out} \alpha_{in}
\mathbf{P}^A_{z \gets y}c_{in},
    \end{align*}
for some volume factor $\alpha_{out} > 0$ independent from the connection $A$.

Recall that the source-to-solution map $L_A$ determines $\sigma[v_{out}](z,\zeta)$ via (\ref{recovery_of_prinsymb}), and observe that both the factors $\alpha_{in}$ and $\alpha_{out}$ are independent from $A$. Therefore $L_A$ determines the parallel transport $\mathbf{P}^A_{z \gets y}c_{in}$.
Letting $x_2,x_3 \to x_1 = x$ and $c_2,c_3 \to c_1$,
we have
$$
\mathbf{P}^A_{z \gets y}c_{in} \to
3 |\mathbf{P}^A_{y \gets x} c_1|^2 \mathbf{P}^A_{z \gets y} \mathbf{P}^A_{y \gets x} c_1 = 3 |c_1|^2 \mathbf{S}^A_{z \gets y \gets x} c_1,
$$
where we used the fact that $\mathbf{P}^A_{y \gets x}$ is unitary. Thus $L_A$ determines $\mathbf{S}^A_{z \gets y \gets x}$ after varying $c_1 \in E_x \setminus 0$, and we have shown Theorem \ref{th_from_L_to_S}.

\section{Inversion of the broken light ray transform}\label{sec_X-ray}

From now on we assume that $E$ is the trivial bundle $M \times \C^n$. Recall the definition (\ref{def_mho}) of $\mho$. More generally, we write  
$$
\mho(\epsilon) = (0,1) \times B(\epsilon)
$$
where $B(\epsilon)$ is the open ball of radius $\epsilon > 0$, centred at the origin of $\R^3$.
We use also the shorthand notation
$$
\mathbb D(\mho(\epsilon)) = \{y \in M : \text{there is $(x,y,z) \in \mathbb S^+(\mho(\epsilon))$}\},
\quad 0 < \epsilon < \epsilon_0,
$$
where $\mathbb S^+(\mho(\epsilon))$ is defined by (\ref{def_diamonds_etc}).

We recall that $\mathbf{S}^A_{z \gets y \gets x}$ is defined by (\ref{def_S}), that is, 
    \begin{align*}
\mathbf{S}^A_{z \gets y \gets x} &= \mathbf{P}^A_{z \gets y} \mathbf{P}^A_{y \gets x}, \quad (x,y,z) \in \mathbb S^+(\mho),
    \end{align*}
and that the parallel transport map $\mathbf{P}^A_{y \gets x}$ is the fundamental solution to the ordinary differential equation (\ref{parallel_transport}).
In this section we will prove the following:

\begin{proposition}\label{prop_S}
Let $A$ and $B$ be two connections in $\R^{1+3}$ such 
that for all $0 < \epsilon < \epsilon_0$ there holds
\begin{equation}\label{S_equal}
\mathbf{S}^A_{z \gets y \gets x} = \mathbf{S}^B_{z \gets y \gets x}, \quad \text{for all $(x,y,z) \in \mathbb S^+(\mho(\epsilon))$}.
\end{equation}
Then there exists a smooth $\mathbf{u}:\mathbb{D}(\Omega(\epsilon_0)) \to U(n)$ such that 
    \begin{align*}
\mathbf{u}|_{\mho(\epsilon_0)}=\id \quad \text{and}  \quad B=\mathbf{u}^{-1}d\mathbf{u}+\mathbf{u}^{-1}A\mathbf{u}.
    \end{align*}
\end{proposition} 

Observe that the causal diamond, defined by (\ref{def_diamond}), satisfies $\mathbb D = \mathbb D(\mho(\epsilon_0))$, and therefore Theorem \ref{thm:main} follows immediately by combining Theorem \ref{th_from_L_to_S} and Proposition \ref{prop_S}. 
Similarly to $\mathbb S^{out}(\mho)$ we define
    \begin{align*}
\mathbb S^{in}(\mho) = \{(x,y) : \text{$(x,y,z) \in \mathbb S^+(\mho)$ for some $z \in \mho$}\}.
    \end{align*}

\begin{lemma}\label{lem_gauge}
Let $A$ and $B$ be two connections in $\R^{1+3}$.
We define $\mathbf{u}(y,x) = \mathbf{P}^A_{y \gets x} \mathbf{P}^B_{x \gets y}$ for $(x,y) \in \mathbb S^{in}(\mho)$.
If (\ref{S_equal}) holds then $\mathbf{u}(y,x_1) = \mathbf{u}(y,x_2)$ for all $(x_j,y) \in \mathbb S^{in}(\mho)$, $j=1,2$.
\end{lemma}
\begin{proof}
Note that $\mathbf{P}^A_{y \gets x}$ is a linear isomorphism and $(\mathbf{P}^A_{y \gets x})^{-1} = \mathbf{P}^A_{x \gets y}$. In particular, 
$$
(\mathbf{S}^A_{z \gets y \gets x})^{-1} = \mathbf{S}^A_{x \gets y \gets z},
$$
and (\ref{S_equal}) implies that 
$\mathbf{S}^A_{x \gets y \gets z} = \mathbf{S}^B_{x \gets y \gets z}$
for all $(x,y,z) \in \mathbb S^+(\mho)$.

Consider now $y$, $x_1$ and $x_2$ as in the claim. 
Then there is $z \in \mho$ such that $(x_j, y, z) \in \mathbb S^+(\mho)$ for $j=1,2$.
We have
$$
\mathbf{S}^A_{x_2 \gets y \gets z} \mathbf{S}^A_{z \gets y \gets x_1}
= \mathbf{P}^A_{x_2 \gets y} \mathbf{P}^A_{y \gets z} \mathbf{P}^A_{z \gets y} \mathbf{P}^A_{y \gets x_1} 
= \mathbf{P}^A_{x_2 \gets y} \mathbf{P}^A_{y \gets x_1}.
$$
But $\mathbf{S}^A_{x_2 \gets y \gets z} \mathbf{S}^A_{z \gets y \gets x_1} = \mathbf{S}^B_{x_2 \gets y \gets z} \mathbf{S}^B_{z \gets y \gets x_1}$, and therefore 
$
\mathbf{P}^A_{x_2 \gets y} \mathbf{P}^A_{y \gets x_1} = \mathbf{P}^B_{x_2 \gets y} \mathbf{P}^B_{y \gets x_1}.
$
We apply $\mathbf{P}^A_{y \gets x_2}$ on left and $\mathbf{P}^B_{x_1 \gets y}$ on right, and obtain 
    \begin{align*}
\mathbf{P}^A_{y \gets x_1} \mathbf{P}^B_{x_1 \gets y} = \mathbf{P}^A_{y \gets x_2} \mathbf{P}^B_{x_2 \gets y}.
    \end{align*}
\end{proof}

\begin{proof}[Proof of Proposition \ref{prop_S}]
For $y \in \mathbb D(\mho)$ there are $x,z \in \mho$ such that $(x,y,z) \in \mathbb S^+(\mho)$
and we define $\mathbf{u}(y,x)$ as in Lemma \ref{lem_gauge}. It follows from Lemma \ref{lem_gauge} that
$\mathbf{u}(y,x) = \mathbf{u}(y)$ and $\mathbf{u}$ can be viewed as a function of
$y \in \mathbb D(\mho)$. The parallel transport map 
takes values in $U(n)$ and therefore $\mathbf{u}$ is a section of $U(n)$ over $\mathbb D(\mho)$. 

Observe that $\mho \cap \mathbb D(\mho) = \emptyset$ by the definition of $\mathbb S^+(\mho)$, see (\ref{def_diamonds_etc}).
For this reason we shrink $\mho$, that is, we
define $\mathbf{u}$ as above but replace $\mho$ with $\mho(\epsilon)$, $0 < \epsilon < \epsilon_0$. This allows us to define $\mathbf{u}$ on $\mho(\epsilon_0) \setminus \mu([0,1])$, see (\ref{def_mu}) for the definition of the path $\mu$. By the continuity of the parallel transport map, we can define $\mathbf{u}$ on the whole $\mho(\epsilon_0)$.
Using the continuity again, we let $x \to y \in (0,1) \times \p B(\epsilon)$ in $\mathbf{u}(y,x)$, and see that $\mathbf{u}(y) = \id$
for any $y \in (0,1) \times \p B(\epsilon)$ and any $0 < \epsilon < \epsilon_0$. Using the continuity once more, we see that $\mathbf{u} = \id$ in the whole $\mho = \mho(\epsilon_0)$.

We write $\tilde A = \mathbf{u}^{-1}d\mathbf{u}+\mathbf{u}^{-1}A\mathbf{u}$.
Let $(x,y) \in \mathbb S^{in}(\mho(\epsilon))$, $0 < \epsilon < \epsilon_0$, and consider the geodesic $\gamma_{y \gets x}$ from $x$ to $y$. Let $s_{in} > 0$ satisfy $y = \gamma_{y \gets x}(s_{in})$.
By using the choice $\mathbf{u}(y) = \mathbf{u}(y,x)$ we will show that
\begin{equation}\label{gauge_match}
\langle \tilde A, \dot \gamma_{y \gets x} \rangle = \pair{B, \dot \gamma_{y \gets x}}.
\end{equation}

We write $\gamma = \gamma_{y \gets x}$ and consider the fundamental matrix solution $$\mathbf{U}^A:\R^2 \to U(n)$$ of (\ref{parallel_transport}).
That is, for fixed $s \in \R$, the function $\mathbf{U}^A(t,s)$ in $t$ is the solution of
\begin{equation}\label{fundamental_sol}
\begin{cases}
\p_t \mathbf{U}^A(t,s) + \pair{A, \dot{\gamma}(t)} \mathbf{U}^A(t,s)=0, & \text{on $\R$},
\\
\mathbf{U}^A(s,s)=\id.
\end{cases}
\end{equation}
Clearly $\mathbf{P}_{\gamma}^A =\mathbf{U}^A(s_{in},0)$. 
We define $\mathbf U^B$ analogously.

By (\ref{fundamental_sol}) it holds that 
$$
\p_t \mathbf{U}^A(s,s) = - \pair{A, \dot{\gamma}(s)} \mathbf{U}^A(s,s) 
= -  \pair{A, \dot{\gamma}(s)}.
$$
Moreover, differentiating (\ref{fundamental_sol}) in $s$, gives
$$
\begin{cases}
\p_t \p_s \mathbf{U}^A(t,s) + \pair{A, \dot \gamma(t)} \p_s \mathbf{U}^A(t,s) = 0,
\\
\p_t \mathbf{U}^A(s,s) + \p_s \mathbf{U}^A(s,s) = 0.
\end{cases}
$$
Therefore $\mathbf{W}(t,s) = \p_s \mathbf{U}^A(t,s)$ satisfies
$$
\begin{cases}
\p_t \mathbf{W}(t,s) + \pair{A, \dot \gamma(t)} \mathbf{W}(t,s) = 0,
\\
\mathbf{W}(s,s) = \pair{A, \dot{\gamma}(s)},
\end{cases}
$$
and writing $\mathbf{W}$ in terms of the fundamental solution $\mathbf{U}^A$ gives
\begin{equation}\label{ps_U}
\p_s \mathbf{U}^A(t,s) = \mathbf{U}^A(t,s) \pair{A, \dot{\gamma}(s)}.
\end{equation}

We have $\mathbf{u}(\gamma(t),y) = \mathbf{U}^A(t, 0) \mathbf{U}^B(0, t)$, and 
using (\ref{fundamental_sol}) for $A$ and (\ref{ps_U}) for $B$,
\begin{align*}
\pair{d\mathbf{u}, \dot \gamma(t)}
&= \p_t ( \mathbf{U}^A(t, 0) \mathbf{U}^B(0, t) )
\\&= - \pair{A, \dot{\gamma}(t)} \mathbf{U}^A(t,0) \mathbf{U}^B(0, t) +
\mathbf{U}^A(0,t) \mathbf{U}^B(0, t) \pair{B, \dot{\gamma}(t)}
\\&= - \pair{A \mathbf{u}, \dot{\gamma}(t)} + \pair{\mathbf{u} B, \dot{\gamma}(t)}.
\end{align*}
We obtain (\ref{gauge_match}) after rearranging and multiplying both sides with $\mathbf{u}^{-1}$,
$$
\pair{\mathbf{u}^{-1} d\mathbf{u} + \mathbf{u}^{-1} A \mathbf{u}, \dot \gamma(t)}
= \pair{B, \dot{\gamma}(t)}.
$$

We have shown that (\ref{gauge_match}) holds for any $(x,y) \in \mathbb S^{in}(\mho(\epsilon))$.
As $\mho(\epsilon)$ is open, it follows from Lemma \ref{Lauri's Lemma}, that the vectors $\dot \gamma_{y \gets x}$ span $T_y M$ as $x$ varies in the set $\{x \in \mho : (x,y) \in \mathbb S^{in}(\mho(\epsilon)) \}$.
Hence $\tilde A = B$ at $y$ for each $y \in \mathbb D(\mho(\epsilon))$ and each $0 < \epsilon < \epsilon_0$.
\end{proof}

\begin{appendix}

\section{Conormal distributions and IPL distributions}
\label{appendix}

We formulate the framework of $E \otimes \Omega^{1/2}$ valued conormal and IPL distributions for the application to the connection wave equation. 
The contents of the appendix is modelled on the pioneering works of H\"{o}rmander \cite{Hormander-FIO1}, Duistermaat-H\"ormander \cite{Duistermaat-Hormander-FIO2} and Melrose-Uhlmann \cite{Melrose-Uhlmann-CPAM1979}.  Additionally, we refer the reader to H\"{o}rmander's books \cite{Hormander1, Hormander-Vol3} for the basics of distributions and half densities.

Let us begin with the notion of conormal distribution.

\begin{definition}
Let $X$ be an $n$-dimensional manifold and $Y$ an $(n - N)$-dimensional closed submanifold. We say a section-valued distribution	$u \in \mathcal{D}^\prime(X, E \otimes \Omega^{1/2})$ is a member of the conormal distributions $I^m(X; N^\ast Y; E \otimes \Omega^{1/2})$, if in local coordinates $(x', x'') \in \R^{N + (n-N)}$ on $X$, such that $Y$ is defined by $x' = 0$,
the distribution $u$ 
takes the following form 
    \begin{align}\label{def_conorm_dist}
u = (2\pi)^{-(n + 2N)/4}e^{- \imath \pi N/4}\int_{\mathbb{R}^N} e^{i \theta x'} a(x, \theta) \,d\theta,
    \end{align}
where $a \in S^{m + n/4}(\R^n \times (\R^N \setminus 0); E\otimes \Omega^{1/2})$.
\end{definition}	

From the viewpoint of Lagrangian distributions, $u$ is associated with the conormal bundle $N^\ast Y \setminus 0$ and $\text{WF}(u) \subset N^\ast Y \setminus 0$. If $\xi = (\xi', \xi'')$ denotes the induced coordinates on the cotangent space of $X$,  $N^\ast Y$ is defined by $\{x' = 0, \xi''= 0\}$. The principal symbol $\sigma[u] \in S^{m+n/4}(N^\ast Y \setminus 0; E \otimes \Omega^{1/2})$ is defined as $\sigma[u](x'',\xi') = a(0, x'', \xi')$. We have the following exact sequence
\begin{equation*}
\begin{split}\lefteqn{
0 \hookrightarrow I^{m - 1}(X; N^\ast Y \setminus 0; E\otimes \Omega^{1/2} )    \hookrightarrow I^{m}(X; N^\ast Y \setminus 0; E\otimes \Omega^{1/2} ) } \\ & \quad\quad\quad \quad\quad\quad \quad\quad\quad   \xrightarrow{\sigma} S^{m + n/4}( N^\ast Y \setminus 0; E\otimes \Omega^{1/2} ) \rightarrow 0.\end{split}
\end{equation*}



When the source term for a linear wave equation is a conormal distribution, the solution is not in the same class, and the wider class of IPL distributions was introduced in \cite{Melrose-Uhlmann-CPAM1979} to tackle this problem.

We begin with the model case where $\tilde{\Lambda}_0 = T_0^\ast \mathbb{R}^n \setminus 0$ and   
$$\tilde{\Lambda}_1 = \{(x, \xi) \in T^\ast \mathbb{R}^n \setminus 0 : \xi_1 = 0,\ x^1 \geq 0,\ x^2 = \cdots = x^n = 0\}.$$
We write $x = (x^1, x')$ and $\xi = (\xi_1, \xi')$.

\begin{definition}
We say that a compactly supported distribution $u \in \mathcal{E}'(\mathbb{R}^n; E \otimes \Omega^{1/2})$ belongs to the space $I^m(\mathbb{R}^n; \tilde{\Lambda}_0, \tilde{\Lambda}_1; E \otimes \Omega^{1/2})$, if modulo compactly supported smooth functions, it can be expressed as $$
u= \int_0^\infty \int_{\mathbb{R}^n} \exp (\imath \xi_1 (x^1 - s) + \xi'x')\, a(s, x, \xi) \,d\xi ds,$$ where the amplitude $a \in S^{m + 1/2 - n/4}(\R^n \times (\R^n \setminus 0); E\otimes \Omega^{1/2})$ is supported on the conical closure of a compact set.
\end{definition}

Observe that $\tilde \Lambda_0 = N^* \{0\} \setminus 0$,
and that $\tilde \Lambda_1 \setminus \p \tilde \Lambda_1 \subset N^* Y \setminus 0$ where 
$Y = \{x' = 0\}$.
The wavefront set of $A$ is contained in the union $\tilde{\Lambda}_0 \cup \tilde{\Lambda}_1$ and away from the intersection $\p \tilde\Lambda_1 = \tilde\Lambda_0 \cap \tilde\Lambda_1$, the IPL distribution $A$ is a conormal distribution in the following sense, see \cite[pp. 486--487 Proposition 2.3 and Remark 2.7]{Melrose-Uhlmann-CPAM1979}. 

\begin{proposition}\label{prop : model case symbol map}Suppose $u \in I^m(\mathbb{R}^n; \tilde{\Lambda}_0, \tilde{\Lambda}_1; E \otimes \Omega^{1/2})$ and $\chi$ is a  zero-th order pseudodifferential operator. We have
\begin{eqnarray*}
&\chi u \in I^m(\mathbb{R}^n; N^* Y \setminus 0; E\otimes\Omega^{1/2})& \mbox{when $\WF(\chi) \cap  \tilde{\Lambda}_0 = \emptyset$};\\
& \chi u \in I^{m - 1/2}(\mathbb{R}^n;  N^* \{0\} \setminus 0; E\otimes\Omega^{1/2})& \mbox{when $\WF(\chi) \cap  \tilde{\Lambda}_1 = \emptyset$}.\end{eqnarray*}
This microlocalization leads to the following local principal symbol maps, ignoring $2 \pi$ and $\imath$ factors related to the normalization in (\ref{def_conorm_dist}),
\begin{eqnarray}&\sigma^{(1)}[u] = a(x^1, (x^1, 0), (0, \xi'))& \mbox{on $\tilde{\Lambda}_1 \setminus \partial \tilde{\Lambda}_1$};\nonumber \\ &\sigma^{(0)}[u] = -\imath a(0, 0, \xi)/\xi_1& \mbox{on $\tilde{\Lambda}_0 \setminus \partial \tilde{\Lambda}_1$};\nonumber \\ &\xi_1\sigma^{(0)}[u] = -\imath \sigma^{(1)}[u]& \mbox{on $ \partial \tilde{\Lambda}_1$}. \label{eqn : map R(model case)}\end{eqnarray}\end{proposition}

IPL distributions can be defined on any pair of Lagrangians $(\Lambda_0, \Lambda_1)$, $\Lambda_j \subset T^* X \setminus 0$, with a clean intersection. By a clean intersection of two Lagrangians, we mean $$T_\lambda(\Lambda_0) \cap T_\lambda(\Lambda_1) = T_\lambda(\partial \Lambda_1)\quad \mbox{for any $\lambda \in \partial \Lambda_1$},$$ given two Lagrangians $\Lambda_0$ and $\Lambda_1$ with $\Lambda_0 \cap \Lambda_1 = \partial \Lambda_1$.
However, we will omit discussion of Lagrangian distributions, and make the additional assumption that 
    \begin{align}\label{Lambda_conorm}
\Lambda_j \setminus \p \Lambda_1 \subset N^* Y_j \setminus 0, \quad j=0,1,
    \end{align}
for some submanifolds $Y_j \subset X$.

\begin{definition}Suppose $(\Lambda_0, \Lambda_1)$, a pair of Lagrangians over a smooth $n$-manifold $X$, intersect cleanly at $\Lambda_1$.  The space $I^m(X; \Lambda_0, \Lambda_1; E \otimes \Omega^{1/2})$ consists of distributions of the form $$u_0 + u_1 + \sum_j F_j v_j,$$ where
$u_0 \in I^{m -1/2}(X; \Lambda_0;  E \otimes \Omega^{1/2})$;
$u_1 \in I^{m}(X; \Lambda_1 \setminus \partial \Lambda_1;  E \otimes \Omega^{1/2})$;
$\{F_j\}$ is a family of zero-th order Fourier integral operators associated with the inverse of the homogeneous symplectic transformation from $V_j$ to $T^\ast \mathbb{R}^n$, where $\{V_j\}$ is a locally finite, countable covering of $\partial\Lambda_1$;
and $v_j \in I^m(\mathbb{R}^n; \tilde{\Lambda}_0, \tilde{\Lambda}_1; E \otimes \Omega^{1/2})$.
\end{definition}

To symbolically construct the parametrix of the wave operator, the key thing is to understand the principal symbols of IPL distributions. As in the model case, away from the intersection of $\Lambda_0$ and $\Lambda_1$, the corresponding IPL distributions are conormal distributions assuming (\ref{Lambda_conorm}).

\begin{proposition}Suppose that (\ref{Lambda_conorm}) holds. Let $u \in I^m(X; \Lambda_0, \Lambda_1; E \otimes \Omega^{1/2})$ and let $\chi$ be a properly supported zero-th order pseudodifferential operator. We have
\begin{eqnarray*}&\chi u \in I^m(X; N^* Y_1 \setminus 0; E\otimes\Omega^{1/2})& \mbox{when $\WF(\chi) \cap  \Lambda_0 = \emptyset$};\\
&\chi u \in I^{m - 1/2}(X; N^* Y_0 \setminus 0; E\otimes\Omega^{1/2})& \mbox{when $\WF(\chi) \cap  \Lambda_1 = \emptyset$}.\end{eqnarray*}
\end{proposition}

Choosing $\chi$ so that $\sigma[\chi](x,\xi) \ne 0$
at $(x,\xi) \in (N^* Y_0 \setminus 0 \cup N^* Y_1 \setminus 0) \setminus \p \Lambda_1$, 
we can define the principal symbols of $u$ away from the intersection $\p \Lambda_1$,
\begin{eqnarray*}&&\sigma^{(1)}[u] = \sigma[\chi u] / \sigma[\chi] \in S^{m+n/4} (N^* Y_1 \setminus 0; E \otimes \Omega^{1/2});\\
&&\sigma^{(0)}[u] = \sigma[\chi u] / \sigma[\chi]  \in S^{m-1/2+n/4} (N^* Y_0 \setminus 0; E \otimes \Omega^{1/2}).\end{eqnarray*}
However, the map from the IPL distributions $I^m(X; \Lambda_0, \Lambda_1; E \otimes \Omega^{1/2}) $ to the symbols
$$
S^{m+n/4} (\Lambda_1 \setminus \partial \Lambda_1; E \otimes \Omega^{1/2}) \times S^{m-1/2+n/4} (\Lambda_0 \setminus \partial \Lambda_1; E \otimes \Omega^{1/2})
$$ can not be defined as the principal symbol map, since it is not surjective. 
Indeed, Proposition \ref{prop : model case symbol map} implies that $\sigma^{(1)}[u]$ extends to a smooth section of $E \otimes \Omega^{1/2}$ up to $\partial \Lambda_1$, whilst $h\sigma^{(0)}[u]$ extends to a smooth section over $\Lambda_0$ if $h$ is a smooth function vanishing on $\partial \Lambda_1$. Hence the principal symbol space should be a proper subspace of the product space of symbols.

Moreover, the fact that the Lagrangians $\Lambda_j$, $j=0,1$, may not be conormal bundles near the intersection $\p \Lambda_1$ causes additional complications. In the case of applications that we are interested in, $\Lambda_1$ fails to be a conormal bundle near $\p \Lambda_1$.
The principal symbol of a general Lagrangian distribution is not a section of $E \otimes \Omega^{1/2}$ but a section of $E \otimes \Omega^{1/2} \otimes L$ where $L$ is the Maslov bundle. To avoid discussion of Maslov bundles over $\Lambda_j$, $j=0,1$, we make the assumption that they are trivial. In all the cases that we are interested in, $\Lambda_0$ is a conormal bundle and $\Lambda_1$ is the future flowout of $\Lambda_0 \cap \Sigma(P)$, with $P$ the wave operator in the Minkowski space. We will show that the Maslov bundles over the flowouts of interest are trivial in Section \ref{sec_Maslov_flowout} below. 

To understand the principal symbol space, assuming triviality of the Maslov bundles, we now elucidate the relationship between $\sigma^{(1)}[u]$ and $\sigma^{(0)}[u]$. Following \cite{Melrose-Uhlmann-CPAM1979}, we define a map
$$
R: S^{m - 1/2 + n/4}(\Lambda_0 \setminus \partial \Lambda_1; E \otimes \Omega^{1/2}) \to S^{m + n/4}(\Lambda_1; E \otimes \Omega^{1/2})|_{\partial\Lambda_1},
$$ 
that encodes the relation (\ref{eqn : map R(model case)}).
Then the sections,
$$\left\{(a^{(1)}, a^{(0)}) \,\bigg|\begin{array}{l}  a^{(0)} \in S^{m - 1/2 + n/4}(\Lambda_0 \setminus \partial \Lambda_1; E \otimes \Omega^{1/2}),\\a^{(1)} \in S^{m + n/4}( \Lambda_1; E \otimes \Omega^{1/2}), \\ a^{(1)}|_{\partial \Lambda_1} = \mathscr R a^{(0)},\\ \mbox{$ha^{(0)}$ is smooth on $\partial\Lambda_1$ if $h$ vanishes on $\partial\Lambda_1$}. \end{array}\right\},$$
where 
    \begin{align}\label{def_R}
\mathscr R = e^{\imath \pi/4} (2\pi)^{1/4} R,
    \end{align}
will be the desired principal symbol space, that we will call $S^m(\Lambda_1, \Lambda_0; E \otimes \Omega^{1/2})$, and we have the exact sequence \begin{equation}\label{eqn : exact sequence}\begin{split}\lefteqn{
0 \hookrightarrow I^{m - 1/2}(X; \Lambda_0; E\otimes \Omega^{1/2} ) + I^{m - 1}(X; \Lambda_0, \Lambda_1; E\otimes \Omega^{1/2} )} \\ &\quad\quad\quad\hookrightarrow I^{m}(X; \Lambda_0, \Lambda_1; E\otimes \Omega^{1/2} ) \xrightarrow{\sigma} S^{m + n/4}( \Lambda_0, \Lambda_1; E\otimes \Omega^{1/2} ) \rightarrow 0.\end{split}
\end{equation}
We remark that the constant $e^{\imath \pi/4} (2\pi)^{1/4}$ in (\ref{def_R}) results from the fact that $\Lambda_0$ is parametrized with $n + 1$ phase variables whilst  $\Lambda_1$ is parametrized with $n$ phase variables, cf. the normalization in (\ref{def_conorm_dist}).

Apart from the Hermitian bundle factor, the map $R$ was constructed by Melrose-Uhlmann \cite[p.491--493]{Melrose-Uhlmann-CPAM1979}. But that factor is harmless.
Indeed, after passing to the model case, we can simply define $R$ by (\ref{eqn : map R(model case)}), that is,  $R a^{(0)} = \imath (\xi_1 a^{(0)})|_{\p \Lambda_1}$.
This definition entails, of course, that $R$ does not depend on the choice of a homogeneous symplectic transformation that maps $(\Lambda_0, \Lambda_1)$ locally to the model case $(\tilde \Lambda_0, \tilde \Lambda_1)$. Such coordinate invariance was shown in \cite{Melrose-Uhlmann-CPAM1979}.

Having the principal symbol map (\ref{eqn : exact sequence}), we are ready to give a proof of Theorem \ref{th_parametrix}, following \cite{Melrose-Uhlmann-CPAM1979}.

\begin{proof}[Proof of Theorem \ref{th_parametrix}]
We will construct a parametrix for (\ref{eq_wave_lin}),
and the claim will then follow from solving (\ref{eq_wave_lin}) with a smooth function $f$ on the right-hand side. 
The first step of the construction is to find 
$
u_{(0)} \in I^{k - 2 + 1/2}(M; \Lambda_0, \Lambda_1; E\otimes \Omega^{1/2} )
$
such that \begin{equation}\label{eqn : first parametrix}Pu_{(0)} = f + f_{(1)} + e_{(1)},\end{equation} where $f_{(1)} \in  I^{k - 1}(M; \Lambda_0; E\otimes \Omega^{1/2} )$ and $e_{(1)} \in I^{k - 3/2}(M; \Lambda_0, \Lambda_1; E\otimes \Omega^{1/2} )$.

To do so, we use the symbol calculus. First of all, since $P$ is elliptic on $\Lambda_0 \setminus \partial \Lambda_1$, we choose $$\sigma[u_{(0)}] = p^{-1} \sigma[f] \quad \mbox{on $\Lambda_0 \setminus \partial \Lambda_1$},$$ where we used the shorthand notation $p=\sigma[P]$ for the principal symbol of $P$. As an element of $I^{k - 2 + 1/2}(M; \Lambda_0, \Lambda_1; E\otimes \Omega^{1/2} )$, $u_{(0)}$ has the following principal symbol on $\partial\Lambda_1$, \begin{equation}\label{eqn : elliptic construction}
\sigma[u_{(0)}]|_{\partial \Lambda_1} = \mathscr R(p^{-1} \sigma[f]).
\end{equation} This can be viewed as the initial condition of the bicharacteristic flow on $\Lambda_1$ emanating from $\Lambda_0$. On the other hand, the symbol calculus on $\Lambda_1$ obeys the  transport equation, $$\sigma[Pu] = ( \imath^{-1} \mathscr{L}_{H_P} + c) \sigma[u] \quad\mbox{on $\Lambda_1$},$$ where $c = \sigma_{sub}[P]$ is the subprincipal symbol of $P$. Noting that $\WF(f)$ does not intersect $\Lambda_1 \setminus \partial \Lambda_1$, we have
    \begin{align}\label{eqn : first transport equation}
(\mathscr{L}_{H_P} + \imath c) \sigma[u_{(0)}] = 0, \quad\mbox{on $\Lambda_1 \setminus \p \Lambda_1$}.
    \end{align}

Combining \eqref{eqn : elliptic construction} and \eqref{eqn : first transport equation}, we have solved \eqref{eqn : first parametrix}.
Next, we iteratively solve the following equations \begin{equation}\label{eqn : j-th parametrix}Pu_{{(j - 1)}} = f_{(j - 1)} + e_{(j - 1)} + f_{(j)} + e_{(j)},\end{equation} where $f_{(j)} \in  I^{k - j}(M; \Lambda_0; E\otimes \Omega^{1/2} )$ and $e_{(j)} \in I^{k - j - 1/2}(M; \Lambda_0, \Lambda_1; E\otimes \Omega^{1/2} )$. This can be done by choosing $u_{(j)}$ obey \begin{align*}
\sigma[u_{(j)}] &= p^{-1} \sigma[f_{(j)}], &\quad\mbox{on $\Lambda_0$};\\
(\mathscr{L}_{H_P} + \imath c) \sigma[u_{(j)}] &= \imath \sigma[e_{(j)}], &\quad\mbox{\mbox{on $\Lambda_1 \setminus \p \Lambda_1$}} ;\\  
\sigma[u_{(j)}] &= \mathscr R(p^{-1} \sigma[f_{(j)}]), &\quad\mbox{on $\partial\Lambda_1$}.
\end{align*}
We complete the proof by adding up the equations \eqref{eqn : first transport equation} and \eqref{eqn : j-th parametrix} for $j = 1, \cdots, N$, and letting $N \rightarrow \infty$.
\end{proof}

\section{Theorem \ref{th_parametrix} in the context of Example \ref{ex_f_delta}}
\label{appendix2}

Let $\Lambda_j$, $j=0,1$, be as in Example \ref{ex_f_delta}. That is, $\Lambda_0 = N^* \{0\} \setminus 0$ and, taking into account the microlocal cutoff $\chi$,
    \begin{align}\label{Lambda1_ex}
\Lambda_1 = \{(t, t\theta; -\lambda , \lambda \theta) \in \R^{1+3} \times \R^{1+3} : \lambda \in \R \setminus 0,\ \theta \in \mathcal V,\ t \ge 0\},
    \end{align}
where $\mathcal V$ is a small neighbourhood of $\theta_0$ in $S^2$. The smaller $\WF(\chi)$ is, the smaller we can choose $\mathcal V$. 

Observe that $\Lambda_1$ has two components 
$$
\Lambda_1^\pm = \{(t, t\theta; -\lambda , \lambda \theta) \in \R^{1+3} \times \R^{1+3} : \pm\lambda > 0,\ \theta \in \mathcal V,\ t \ge 0\}.
$$
Let us give an explicit choice of two homogeneous symplectic transformations $F^\pm$ taking $(\Lambda_0, \Lambda_1^\pm)$ to the model case $(\tilde \Lambda_0, \tilde \Lambda_1)$.
We define in a neighbourhood of $\Lambda_1^+$ in $T^* \R^{1+3} \setminus 0$ the map 
$$
F^+(x^0,x';\xi_0, \xi') = (x^0, x' - x^0 \xi' / |\xi'|; \xi_0 + |\xi'|, \xi').
$$
Then $F^+$ takes $N^* \{0\} \setminus 0$ to itself and $\Lambda_1^+$ to $\tilde \Lambda_1$. Indeed, 
$$
F^+(0,0;\xi_0, \xi') = (0, 0; \xi_0 + |\xi'|, \xi'),
$$
and for $\lambda > 0$,
$$
F^+(t,t\theta;-\lambda, \lambda \theta) = 
(t,t\theta - t\theta;-\lambda + \lambda, \lambda \theta) = (t, 0; 0, \lambda\theta).
$$
Analogously, we define 
$$
F^-(x^0,x';\xi_0, \xi') = (x^0, x' + x^0 \xi' / |\xi'|; \xi_0 - |\xi'|, \xi')
$$
taking $N^* \{0\} \setminus 0$ to itself and $\Lambda_1^-$ to $\tilde \Lambda_1$.
We will show that $F^-$ is symplectic, the proof for $F^+$ being analogous. 

It holds that 
$$
\frac{dF^-}{d(x^0,x';\xi_0, \xi')}
= 
\begin{pmatrix}
1 & 0 & 0 & 0
\\
\xi'/|\xi'| & \id & 0 & B
\\
0 & 0 & 1 & -(\xi')^T /|\xi'|
\\
0 & 0 & 0 & \id
\end{pmatrix},
$$
where $B$ is a symmetric matrix, the precise form of which is inconsequential. 
Writing shortly $dF^-$ for the above derivative and $J$ for the symplectic form on $T^* \R^{1+3}$ as a matrix, we assert that $(dF^-)^T J dF^- = J$.

We write 
$\eta = \xi'/|\xi'|$ and
$$
J = \begin{pmatrix}
0 & 0 & 1 & 0
\\
0 & 0 & 0 & \id
\\
-1 & 0 & 0 & 0
\\
0 & -\id & 0 & 0
\end{pmatrix}.
$$
Then it follows 
    \begin{align*}
(dF^-)^T J dF^-
&= 
\begin{pmatrix}
1 & \eta^T & 0 & 0
\\
0 & \id & 0 & 0
\\
0 & 0 & 1 & 0
\\
0 & B & -\eta & \id
\end{pmatrix}
\begin{pmatrix}
0 & 0 & 1 & -\eta^T
\\
0 & 0 & 0 & \id
\\
-1 & 0 & 0 & 0
\\
-\eta & -\id & 0 & -B
\end{pmatrix}
\\&=
\begin{pmatrix}
0 & 0 & 1 & -\eta^T + \eta
\\
0 & 0 & 0 & \id
\\
-1 & 0 & 0 & 0
\\
\eta - \eta & -\id & 0 & B- B
\end{pmatrix} = J.
    \end{align*}

Let us now write the initial condition (\ref{transport_init}) in this context. 
We denote by $(\tilde x, \tilde \xi) = (\tilde x^0, \tilde x'; \tilde \xi_0, \tilde \xi')$ the local coordinates on $T^* \R^{1+3}$ given by $F^+$, that is, 
    \begin{align}\label{Fp_coords}
(\tilde x^0, \tilde x'; \tilde \xi_0, \tilde \xi') =  F^+(x^0,x';\xi_0, \xi').
    \end{align}
Recall that $p = \sigma[P]$ is a section of $T^* \R^{1+3} \setminus 0$, and that 
$2 p(x,\xi) = -\xi_0^2 + |\xi'|^2$
in the Cartesian coordinates $(x,\xi) = (x^0,x';\xi_0, \xi')$.
Therefore,
$$
p(\tilde x, \tilde \xi) = 
- (\tilde \xi_0 - |\tilde \xi'|)^2 / 2
+ |\tilde \xi'|^2 / 2
= \tilde \xi_0 (|\tilde \xi'| - \tilde \xi_0 / 2),
$$
and (\ref{transport_init}) can be written
    \begin{align}\label{transp_init_p}
\sigma[u](0, \tilde \xi') &= \mathscr R(p^{-1} \sigma[f])(0, \tilde \xi') = \iota_0 (\tilde \xi_0 (\tilde \xi_0 (|\tilde \xi'| - \tilde \xi_0 / 2))^{-1} \sigma[f](0,\tilde \xi))|_{\tilde \xi_0 = 0}
\\\notag&
= \iota_0 |\tilde \xi'|^{-1} \sigma[f](0; 0, \tilde \xi'),
    \end{align}
where $\iota_0 \in \C \setminus 0$ is a constant containing the $2 \pi$ and $\imath$ factors.
Analogously, setting 
    \begin{align}\label{Fm_coords}
(\tilde x^0, \tilde x'; \tilde \xi_0, \tilde \xi') =  F^-(x^0,x';\xi_0, \xi'),
    \end{align}
we obtain 
$$
p(\tilde x, \tilde \xi) = 
- (\tilde \xi_0 + |\tilde \xi'|)^2 / 2
+ |\tilde \xi'|^2 / 2
= -\tilde \xi_0 (|\tilde \xi'| + \tilde \xi_0 / 2),
$$
and (\ref{transport_init}) reads
    \begin{align}\label{transp_init_m}
\sigma[u](0, \tilde \xi') = -\iota_0 |\tilde \xi'|^{-1} \sigma[f](0; 0, \tilde \xi'),
    \end{align}
with the same $\iota_0$.

\subsection{Trivialization of the Maslov bundle over the flowout}
\label{sec_Maslov_flowout}

It is well-known \cite[Th. 3.3.4]{Hormander-FIO1} that the Maslov bundle over a conormal bundle $N^* K \setminus 0$ is trivial. However, even in the case the flowout $\Lambda_1$ in Example \ref{ex_f_delta}, more care is needed, since $\Lambda_1$ fails to coincide with a conormal bundle near $\mathscr B = \Lambda_1 \cap \Lambda_0$. We will discuss the Maslov bundle in detail only in the context of Example \ref{ex_f_delta}, this being the most important case for the purposes of the present paper. 

When the essential support $\WF(\chi)$ is small around 
$\ccl(0,\pm\xi^0)$, the future flowout $\Lambda_1$ is embedded in the following flowout to both past and future, a localized version of (\ref{Lambda_1_pastfuture}),
$$
\hat \Lambda_1 = \{(t, t\theta; -\lambda , \lambda \theta) \in \R^{1+3} \times \R^{1+3} : \lambda \in \R \setminus 0,\ \theta \in \mathcal V,\ t \in \R\},
$$
where $\mathcal V \subset S^2$ is as in (\ref{Lambda1_ex}). 
We suppose that $\mathcal V$ is chosen so that there is a diffeomorphism $\Theta : B \to \mathcal V$
where $B$ is an open ball in $\R^2$, with the centre at the origin, and $\Theta(0) = \theta_0$.

Analogously to $\Lambda_1$, also $\hat \Lambda_1$ has two components
$$
\hat \Lambda_1^\pm = \{(t, t\theta; -\lambda , \lambda \theta) \in \R^{1+3} \times \R^{1+3} : \pm\lambda >0,\ \theta \in \mathcal V,\ t \in \R\}.
$$
We will show that $\hat \Lambda_1^+$ is contractible. The same holds for $\hat \Lambda_1^-$ with an analogous proof. It then follows that any vector bundle over $\hat \Lambda_1 = \hat \Lambda_1^+ \cup \hat \Lambda_1^-$ is trivial \cite{Bott1982}. 
In particular, the Maslov bundle over $\Lambda_1$ is trivial. As $\Lambda_0 = N^*\{0\} \setminus 0$, also the Maslov bundle over $\Lambda_0$ is trivial. 


Recall that a manifold being contractible means that the identity map is smoothly homotopic to the constant map. 
The manifold $\hat \Lambda_1^+$ is contractible, since the map 
    \begin{align*}
&H : \hat \Lambda_1^+ \times [0,1] \to \hat \Lambda_1^+, \\
&H(t, t\theta; -\lambda , \lambda \theta; s)
= (st, st h_1(s,\theta); -h_2(s,\lambda), h_2(s,\lambda) h_1(s,\theta)),
    \end{align*}
with $h_1(s,\theta) = \Theta(s\Theta^{-1}(\theta))$ and $h_2(s,\lambda) = s\lambda + 1-s$, satisfies
$$
H(t, t\theta; -\lambda , \lambda \theta; 1)
= (t, t\theta; -\lambda , \lambda \theta),
\quad 
H(t, t\theta; -\lambda , \lambda \theta; 0)
= (0, 0, -1, \theta_0).
$$

\section{Flowout from the triple intersection}
\label{appendix3}

We consider the intersection the three cones, defined by  (\ref{def_Kj}), that is,  
    \begin{align*}
K_\mathbf{j} = \{(t_\mathbf{j} + t, x_\mathbf{j}' + t \theta) \in \R^{1+3}: t > 0,\ \theta \in S^2\}, \quad \mathbf{j}=1,2,3,
    \end{align*}
where, for small $r > 0$,
    \begin{align*}
(t_\mathbf{j}, x_\mathbf{j}') &= x_\mathbf{j} = \gamma(-s_{in};y,\xi_\mathbf{j}), \quad 
\xi_\mathbf{j}
= (1,\sqrt{1 - r_\mathbf{j}^2}, r_\mathbf{j}, 0),
\quad r_1 = 0,\ r_2 = r,\ r_3 = -r,
    \end{align*}
see (\ref{xi_eta_form}), (\ref{xi23}) and (\ref{def_xj}).
After a translation, we may assume without loss of generality that $y=0$. Then $t_\mathbf{j} = -s_{in}$ for all $\mathbf{j}=1,2,3$, and
    \begin{align*}
x_\mathbf{j}' = (- s_{in} \sqrt{1 - r_\mathbf{j}^2}, -s_{in} r_\mathbf{j}, 0), \quad \mathbf{j}=1,2,3.
    \end{align*}

The points $(t,x') \in \R^{1+3}$ in $K_1 \cap K_2 \cap K_3$ satisfy
$|x'-x_\mathbf{j}'|^2 = |t+s_{in}|^2$ for each $\mathbf{j}=1,2,3$,
or equivalently, writing $x'=(x^1,x^2,x^3)$,
    \begin{align*}
|x^1 + s_{in} \sqrt{1 - r_\mathbf{j}^2}|^2
+ |x^2 + s_{in} r_\mathbf{j}|^2 + |x^3|^2 = |t+s_{in}|^2,
\quad \mathbf{j}=1,2,3.
    \end{align*}
To simplify the notation, we write $z=x^3$. Taking $x^1=x^2=0$, the above three equations simplify to the single equation $z^2 = t^2 + 2 s_{in} t$. This equation defines the filament in spacetime that acts as an artificial source, as discussed in the introduction.
Solving for $z$ gives an equation for the two moving point sources in Figure \ref{fig_3waves}. In the figure, we have taken $s_{in} = 2$ and $r=0.8$.

For the study of the flowout from the filament it is more convenient to solve for $t$ rather than for $z$. Near $y=0$ this yields
    \begin{align*}
t = T(z) = - s_{in} + \sqrt{s_{in}^2 + z^2}.
    \end{align*}
As $K_1 \cap K_2 \cap K_3$ is a smooth, one dimensional manifold near $y$, it holds locally that
    \begin{align*}
K_1 \cap K_2 \cap K_3
= \{ (T(z), 0, 0, z) : z \in \R \}.
    \end{align*}
    
For a covector $(\zeta_0, \dots,\zeta_4) \in \R^4$ in the fibre of $\Lambda_0 = N^*(K_1 \cap K_2 \cap K_3) \setminus 0$ at $(T(z), 0, 0, z)$ it holds that $\zeta_0 T'(z) + \zeta_4 = 0$ where $T'$ is the derivative of $T$. Therefore
    \begin{align*}
\Lambda_0
= \{ (T(z), 0, 0, z; \zeta_0, \zeta_1, \zeta_2, -\zeta_0 T'(z)) : z \in \R,\ (\zeta_0, \zeta_1, \zeta_2) \in \R^3 \setminus 0 \}.
    \end{align*}
We will proceed to compute the future flowout $\Lambda_1$ of $\Lambda_0 \cap \Sigma(P)$.
To simplify the notation, we write $\zeta_0 = -\lambda$.
Then $(-\lambda, \zeta_1, \zeta_2, \lambda T'(z))$ is lightlike if and only if
    \begin{align*}
\lambda^2 = \zeta_1^2 + \zeta_2^2 + \lambda^2 |T'(z)|^2.
    \end{align*}
We write $(\zeta_1, \zeta_2) = \mu \theta$ with $\mu \ge 0$ and $\theta \in S^1$. Then $\mu^2 = (1-|T'(z)|^2)\lambda^2$, and 
    \begin{align*}
\Lambda_1 \setminus \Lambda_0
= \{ (T(z) + s\lambda, s\lambda \tilde \theta, z + s\lambda \varepsilon; -\lambda, \lambda \tilde \theta, \lambda \varepsilon) :\ 
&\tilde \theta = \sqrt{1-\varepsilon^2} \theta,\ \varepsilon = T'(z),
\\&\theta \in S^1,\ \lambda \in \R \setminus 0,\ z \in \R,\ s \lambda > 0 \}.
    \end{align*}
Observe that $\varepsilon = T'(z) = z / \sqrt{s_{in}^2 + z^2}$,
and this is small when we localize near $y=0$ as the notation suggests. 

Let us show that $\Lambda_1 \setminus \Lambda_0$ coincides with a conormal bundle. It is enough to verify that its projection to the base space $\R^{1+3}$ is a smooth manifold. For our purposes it is enough to consider $\Lambda_1 \setminus \Lambda_0$ only over the compact diamond $\mathbb D$ and therefore it is enough verify that the map $F(t,\theta,z) = (T(z) + t, t \tilde \theta, z + t \varepsilon)$ is injective and has injective differential for $t>0$, $\theta \in S^1$ and small $|z|$.

Suppose that 
    \begin{align}\label{F_inj}
F(t_1,\theta_1,z_1) = F(t_2,\theta_2,z_2).
    \end{align}
As $t> 0$ and $\varepsilon$ is small, the second component of (\ref{F_inj}) implies $\theta_1 = \theta_2$
and $t_1 \sqrt{1-|T'(z_1)|^2} = t_2 \sqrt{1-|T'(z_2)|^2}$. Writing $Z_j = \sqrt{s_{in}^2 + z_j^2}$, and using 
    \begin{align*}
1-|T'(z_j)|^2 
= \frac{s_{in}^2 + z_j^2 - z_j^2}{s_{in}^2 + z_j^2}
= \frac{s_{in}^2}{Z_j^2},
    \end{align*}
the latter equation reduces to 
$t_1 / Z_1 = t_2 / Z_2$.
On the other hand, the last component of (\ref{F_inj})                 
gives
$z_1(1 + t_1 / Z_1) = z_2(1 + t_2 / Z_2)$.
As $1+t_1 / Z_1 = 1 + t_2 / Z_2 > 0$, we get $z_1 = z_2$.
This together with $t_1 / Z_1 = t_2 / Z_2$ implies that also $t_1 = t_2$. We have shown that $F$ is injective.

Observe that, writing $Z = Z_j$ when $z_1 = z$, 
    \begin{align*}
\frac {d \epsilon}{d z} = T''(z) = \frac 1 Z-\frac {z^2} {Z^3} = \frac {1- \varepsilon^2} Z.
    \end{align*}
Therefore, letting $\R \supset B \ni a \mapsto \Theta(a) \in S^1$ be local coordinates on $S^1$,
    \begin{align*}
dF = \begin{pmatrix}
1 & 0 & \varepsilon
\\
\tilde \theta & t \sqrt{1-\varepsilon^2} \Theta'
& - t\frac{\varepsilon}{\sqrt{1-\varepsilon^2}} \frac {1- \varepsilon^2} Z \theta
\\
\varepsilon & 0 & 1 + t \frac {1- \varepsilon^2} Z
\end{pmatrix}.
    \end{align*}
When $z=0$, this reduces to 
    \begin{align*}
dF = \begin{pmatrix}
1 & 0 & 0
\\
\tilde \theta & t \Theta'
& 0
\\
0 & 0 & 1 + \frac {t} Z
\end{pmatrix},
    \end{align*}
which is invertible since $t > 0$, $Z > 0$ and $\Theta' \ne 0$. The same holds when $|z|$ is small enough, and we have shown that $\Lambda_1 \setminus \Lambda_0$ coincides with a conormal bundle.
Let us also remark that an argument similar to that in Section \ref{sec_Maslov_flowout} shows that the Maslov bundle over $\Lambda_1$ is trivial.

\end{appendix}


\bibliographystyle{abbrv}
\bibliography{main}


\end{document}